\documentclass[10pt,a4paper]{article}
\usepackage{amsthm,amssymb,mathrsfs,setspace,amsmath,mathtools,bm,fixmath}
\usepackage[a4paper,bindingoffset=0.2in,
left=0.5in,right=0.5in,
top=1in,bottom=1in,footskip=.25in]{geometry}
\usepackage{color,soul}
\usepackage{enumitem}
\title{Multidimensional Fractional Wavelet Transforms and Uncertainty Principles}
\author{Navneet Kaur$^a$, Bivek Gupta$^b$\thanks{bivekgupta040792@gmail.com}, Amit K. Verma$^c$\thanks{Corresponding author email: akverma@iitp.ac.in}, \\{\small\textit{$^{a,b,c}$Department of Mathematics, IIT Patna, Bihta, Patna, 801106.}}}

\theoremstyle{definition}
\newtheorem{definition}{Definition}[section]
\newtheorem{lemma}{Lemma}[section]

\newtheorem{theorem}{Theorem}[section]
\newtheorem{corollary}{Corollary}[section]

\newcommand{\F}{F}
\newcommand{\p}{\theta}
\newcommand{\q}{p}
\newcommand{\N}{N}
\newcommand{\A}{S}
\begin{document}
\maketitle
\begin{abstract}
In this paper, we have given a new definition of continuous fractional wavelet transform in $\mathbb{R}^{\N},$ namely the multidimensional fractional wavelet transform (MFrWT) and studied some of the basic properties along with the inner product relation and the reconstruction formula. We have also shown that the range of the proposed transform is a reproducing kernel Hilbert space and obtain the associated kernel. We have obtained the uncertainty principle like Heisenberg's uncertainty principle, logarithmic uncertainty principle and local uncertainty principle of the multidimensional fractional Fourier transform (MFrFT). Based on these uncertainty principles of the MFrFT we have obtained the corresponding uncertainty principles i.e., Heisenberg's, logarithmic and local uncertainty principles for the proposed MFrWT.
\end{abstract}
\section{Introduction}
In 1980 Namias introduced the fractional Fourier transform (FrFT), also known as essentially equivalent transforms, as a generalization of the traditional Fourier transform (FT)(\cite{namias1980fractional}) . This theory was later refined and studied in \cite{almeida1994fractional}, \cite{mcbride1987namias}. FrFT has been used as a substantial tool for analysis of the signal, sensor data, transmission signals such as radio signals and many others, for example \cite{almeida1994fractional},\cite{bultheel2002shattered},\cite{kutay1997optimal},\cite{shi2012uncertainty},\cite{ozaktas1999introduction}. Because of the scarcity of localization information, It is not suitable for processing signals with varying fractional frequencies over time. The Fourier transform, on the other hand, has been found to be unsuitable for characterizing some practical uses or dealing with their inherent mathematical issues. As a repercussion, some Fourier transforms have been developed to compensate for the FT's imprecision, such as the FrFT, wavelet transform (\cite{debnath2002wavelet}) and the windowed Fourier transform (\cite{grochenig2001foundations}). Mendlovich et al. (\cite{mendlovic1997fractional}) devised the Fractional Wavelet Transform to analyze with optical signals in 1997. The proposal was to use the wavelet transform of the signal's fractional spectrum, which was obtained using FrFT. Since then, FrFT has produced fractional frequencies that last for the entire duration of the signal rather than a specific time, preventing the signal from communicating local information. As a result, the fractional wavelet transformation investigated in  \cite{mendlovic1997fractional} fails to capture the signal's local properties. As a result, Shi et al. (\cite{shi2012novel}) proposed a new fractional wavelet transform that combines aspects of both the classical and FrFT wavelet transforms and returns the signal's local information. In \cite{dai2017new}, Dai et al. studied a new FrWT, which is more general than the transform studied in \cite{shi2012novel}, \cite{prasad2014generalized}. They also studied the multiresolution analysis associated with it. Luchko et al. (\cite{luchko2008fractional}) gave a novel definition of FrFT and the corresponding fractional wavelet transform have been studied in \cite{srivastava2019certain},\cite{verma2021certain}. For more information on the fractional Fourier transform introduced in \cite{luchko2008fractional}, we refer the reader to \cite{kilbas2010fractional},\cite{srivastava2017family}.

In recent times, Zayed \cite{zayed2018two} proposed 2-dimensional FrFT $$\mathfrak{F}_{\alpha,\beta}(u,v)=\displaystyle \int_{
\mathbb{R}^2}k(x,y,u,v;\alpha,\beta)f(x,y)dxdy,$$ where $k(x,y,u,v;\alpha,\beta)$ is given by equation $(3.13)$ in \cite{zayed2018two}, which is not a tensor product of two 1-dimensional FrFT (\cite{almeida1994fractional}). He verified its properties with  the convolution theorem, the inverse theorem, and the Poisson summation formula. Kamalakkanan et al. \cite{kamalakkannan2020multidimensional} proved the inverse formula for MFRFT (\cite{ozaktas2001fractional})  which is defined as the tensor product of {\N}-copies of a 1-dimensional FrFT. He also studied the related convolution theorem and product theorem, as well as presenting a generalised fractional convolution that was more general than the one in \cite{i1998fractional}.  Verma et al. (\cite{verma2021note}) extended the continuous fractional wavelet transforms in $\mathbb{R}^{\N}$ with the dilation parameter in $\mathbb{R}^{\N}$ and studied the associated uncertainty principles along with its boundedness on Morrey space. As per taking the context on multidimensional fractional Fourier transform and one-dimensional fractional wavelet transforms, in this paper, we provide definition of multidimensional fractional wavelet transforms with ${\N}-$dimensional parameter $\bm{\alpha}$ in a more precise way, including properties of continuous fractional wavelet transforms.

The signal's frequency and time at any point in the time-frequency plane are unknown. To put it another way, we have no way of knowing which spectral components are present at any given time. All we can do is look at which spectrum components are present at any given time frame. This issue is called the uncertainty principle. Heisenberg discovered and formulated the uncertainty principle, which states that a moving particle's momentum and state cannot be determined simultaneously. The principles of logarithmic, Heisenberg, and local uncertainty help us to perceive the interrelationships between different transformed domains better than the unrelated appearance \cite{dai2017new}, \cite{yang2013mathematical}. Because they are signal processing elements, the three uncertainty principles are well suited for potential 
later applications (\cite{guanlei2009logarithmic}). The detailed description and history of these three inequalities are given in \cite{folland1997uncertainty}.

The purpose of this paper is to define novel multidimensional fractional wavelet transform (MFrWT) with parameter $(\alpha_1,\alpha_2,\cdots,\alpha_{\N})$ that is broader in scope than the transforms defined earlier. We give some basic properties of the suggested transform and obtain the inner product relation, reconstruction/retransformation formula and also characterize its range. We derive the Heisenberg's uncertainty inequality, local uncertainty inequality and  logarithmic uncertainty inequality for the MFrFT. Based on the properties of the MFrWT and the uncertainty inequalities associated with the MFrFT we derive the same for the MFrWT.
The rest of the paper is organized as follows. In section 2, we review some fundamental definitions. In section 3, The MFrWT's theoretical framework, including its definition, properties, and inverse transformation, has been established. In addition, we defined the range of transformations and demonstrated that the range is the reproducing  kernel  Hilbert space. In section 4, we have obtained several uncertainty principles for the MFrWT. Lastly, in section 5, we conclude our paper.
   
\section{Preliminaries}
Let $\|\cdot\|$ denote the Euclidean norm and $\mathbb{R}^{\N}$ be the $\N$-dimensional Euclidean space that is, for $\mathbf{x} = (x_1, x_2,\cdots,x_{\N})\in \mathbb{R}^{\N}$
\begin{align*}
\|\mathbf{x}\|=\sqrt{\sum_{i=1}^{\N}x_{i}^2}.
\end{align*}
We specify $|\mathbf{x}|_m=|x_1x_2x_3\cdots x_{\N}|$ and $\mathbb{R}_0^{\N}=\{\mathbf{x}\in\mathbb{R}^{\N}:|\mathbf{x}|_m\neq0\}$. For $\mathbf{x}=(x_1,x_2,\cdots,x_{\N}),~\mathbf{y} = (y_1, y_2,\cdots,y_{\N})\in \mathbb{R}^{\N},~\mathbf{x}+\mathbf{y}=(x_1+y_1,x_2+y_2,\cdots,x_{\N}+y_{\N}),~\textbf{xy}=(x_1y_1,x_2y_2,\cdots,x_{\N}y_{\N}).$ Moreover $y\in \mathbb{R}_0^{\N}$ then, $ \frac{\mathbf{x}}{\mathbf{y}}=\left(\frac{x_1}{y_1},\frac{x_2}{y_2},\cdots,\frac{x_{\N}}{y_{\N}}\right)$ and if $\bm\alpha=(\alpha_1,\alpha_2,\cdots,\alpha_{\N})$ then we define $\sin\bm\alpha=(\sin\alpha_1,\sin\alpha_2,\cdots,\sin\alpha_{\N}).$
\begin{definition}  For $1\leq P < \infty,$ the Lebesgue space $L^P(\mathbb{R}^{\N})$ is a Banach space and for every complex valued measurable function  $f \in \mathbb{R}^{\N}$ such that
\begin{eqnarray*}
\int_{\mathbb{R}^{\N}}|f(\bm{t})|^P d\bm{t}< \infty,
\end{eqnarray*}
where norm is given by
\begin{eqnarray*}
\|f\|_{L^P(\mathbb{R}^{\N})}=\left(\int_{\mathbb{R}^{\N}}|f(\bm{t})|^P d\bm{t}\right)^{\frac{1}{P}}.
\end{eqnarray*}
In particular, $L^2(\mathbb{R}^{\N})$ is a Hilbert space in which the  inner product is stated by
\begin{eqnarray*}
\langle f,g\rangle _{L^2(\mathbb{R}^{\N})}=\int_{\mathbb{R}^{\N}}f(\bm{t})\overline{g(\bm{t})}d\bm{t},
\end{eqnarray*}
here $\overline{g(\bm{t})}$ denotes the complex conjugate of $g(\bm{t}).$
\end{definition}
\begin{definition}
Assume $X$ is a complete inner product space of complex-valued functions defined on $\A,$ where $\A$ is an arbitrary set with the inner product defined by $\langle.,.\rangle _X$. Then a complex-valued function $\mathcal{K}$ defined on $\A \times \A$ is known as reproducing kernel of $X$ if it fulfill the necessary conditions,

for a particular $\q \in \A$, we've $\mathcal{K}(\cdot, \q)$also in  $X$ and $f(\q)=\langle f(\cdot), \mathcal{K}(\cdot,\q)\rangle_X$ for every $f \in X.$
\end{definition}
We'll go through the definition of the MFrFT (\cite{kamalakkannan2020multidimensional}) in the next section. 
\begin{definition}
The MFrFT of $f\in L^2(\mathbb{R}^{\N}),$ of order $\bm\alpha=(\alpha_1, \alpha_2,\cdots,\alpha_{\N}),$ ${\alpha_i} \in \left(-\pi,\pi\right)\backslash\{0\},$ for $i= 1,2,\cdots, {\N}$ and $\lambda\in \mathbb{R}- \{0\},$ is given by
\begin{eqnarray}\label{P1eqn1}
\mathfrak{F}_{\bm{\alpha},\lambda}\left(f\right)\left(\bm{\xi}\right)=F_{\bm{\alpha},\lambda}\left(\bm{\xi}\right)= \int_{\mathbb{R}^{\N}} f\left(\mathbf{x}\right)K_{\bm{\alpha},\lambda}\left(\mathbf{x},\bm{\xi}\right) \,d\mathbf{x},
\end{eqnarray}
where $ K_{{\bm{\alpha}},\lambda}\left(\bm{x},\bm{\xi}\right)= \displaystyle \prod_{i=1}^{\N}K_{\alpha_{k},\lambda}\left(x_{k},\xi_{k}\right)$ and $K_{\alpha_{k},\lambda}\left(x_{k},\xi_{k}\right),~\mbox{for}~k=1,2,\cdots,{\N},$ are defined by
\begin{eqnarray*}
K_{\alpha_{k},\lambda}\left(x_{k},\xi_{k}\right)=\begin{cases}
\frac{c(\alpha_{k})}{(\sqrt{2\pi})}e^{i\lambda^2\{a(\alpha_{k})[{x_{k}}^2+{\xi_{k}}^2-2b(\alpha_{k})x_{k}\xi_{k}]\}},  \alpha_{k}\notin \pi\mathbb{Z}\\
\Delta(x_{k}-\xi_{k}),\quad \quad \quad \quad\quad\quad\quad\quad\quad\quad \alpha_{k}\in 2\pi \mathbb{Z}\\
\Delta(x_{k}+\xi_{k}),\quad \quad\quad\quad\quad\quad\quad\quad\quad\quad  \alpha_{k}\in 2\pi \mathbb{Z}+\pi.
\end{cases}
\end{eqnarray*}
Here, $\mathbf{x}=(x_{1},x_{2}, x_{3},...,x_{\N}),\quad a(\alpha_{k})=\frac{\cot(\alpha_{k})}{2}, \quad b(\alpha_{k})=\sec(\alpha_{k}),\quad c(\alpha_{k})=\sqrt{1-i\cot\alpha_{k}}.$

The kernel $K_{{\bm{\alpha}},\lambda}\left(\mathbf{x},\bm{\xi}\right)$ can be re-written as
\begin{eqnarray*}
K_{\bm{\alpha},\lambda}\left(\mathbf{x},\bm{\eta}\right)=\frac{c(\bm{\alpha}_\lambda)}{(\sqrt{2\pi})^{\N}}e_{\bm{\alpha},\lambda^2}(\mathbf{x})e_{\bm{\alpha},\lambda^2}(\bm{\eta})e^{-i\lambda^2\sum_{k=1}^{\N}x_{k}\eta_{k}\csc\alpha_{k}},
\end{eqnarray*}
where
\begin{eqnarray}
e_{\bm{\alpha},\lambda^2}(x)&=&e^{i\lambda^2 \sum_{k=1}^{\N}a(\alpha_{k}){x_{k}}^2},\quad c(\bm\alpha_\lambda)=c(\alpha_1)c(\alpha_2) \cdots c(\alpha_{\N})
\end{eqnarray}
and
$$ \bm\alpha_\lambda = \left(\cot^{-1}(\lambda^2\cot\alpha_1),\cot^{-1}(\lambda^2\cot\alpha_2),\cdots,\cot^{-1}(\lambda^2\cot\alpha_{\N})\right).$$
\end{definition}

\section{Multidimensional Fractional Wavelet Transform (MFrWT)}\begin{definition}
A non-zero function  $\psi\in L^2(\mathbb{R}^{\N})$ is  wavelet admissible if
\begin{eqnarray}
\displaystyle\int_{\mathbb{R}^{\N}_{0}}|(\mathfrak{F}_{\bm{\alpha},\lambda}\psi)(\bm{u})|^2\frac{d\bm{u}}{|\bm{u}|_m}<\infty.
\end{eqnarray}

A wavelet admissible function is also known as a fractional wavelet or simply a wavelet.
\end{definition}
Now, we define our new MFrWT.

Let $f \in L^2(\mathbb{R}^{\N})$ and $\psi$ be an admissible wavelet, then  MFrWT is given by
\begin{eqnarray*}
(W_{\psi}^{\bm\alpha,\lambda}f)(\bm{a},\bm{b})=c(\bm\alpha_\lambda)e_{\bm{\alpha},-\lambda^2}(\tilde{f} \ast \breve{g})(\bm{b}), \quad \bm{a}\in \mathbb{R}^{\N}_0,\quad \bm{b}\in\mathbb{R}^{\N},
\end{eqnarray*}
where  $\ast$ is the convolution given by
$$(\mu\ast \nu)(\mathbf{x})=\int_{\mathbb{R}^{\N}}\mu(\mathbf{y})\nu(\mathbf{x}-\mathbf{y})d\mathbf{y},$$
$\tilde{\psi}=\psi{e}_{\bm{\alpha},\lambda^2},$ $\tilde{f}=f{e}_{\bm{\alpha},\lambda^2}$ and
${\breve{g}}(\bm{t})=\overline{\frac{1}{\sqrt{|\bm{a}|_m}}\tilde{\psi}\left(\frac{-\bm{t}}{\bm{a}}\right)}.$\\
Therefore,
\begin{eqnarray*}
(W_{\psi}^{\bm\alpha,\lambda}f)(\bm{a},\bm{b})&=&c(\bm\alpha_\lambda)e_{\bm{\alpha},-\lambda^2}(\bm{b})(\tilde{f}*{\breve{g}})(\bm{b})\\
&=&c(\bm\alpha_\lambda)e_{\bm{\alpha},-\lambda^2}(\bm{b})\int_{\mathbb{R}^{\N}}\tilde{f}(\bm{t})\overline{\frac{1}{\sqrt{|\bm{a}|_m}}\tilde{\psi}\left(\frac{\bm{t}-\bm{b}}{\bm{a}}\right)}d\bm{t}\\
&=&c(\bm\alpha_\lambda)e_{\alpha,-\lambda^2}(\bm{b})\int_{\mathbb{R}^{\N}}f(\bm{t})e_{\bm{\alpha},\lambda^2}(\bm{t})\overline{\frac{1}{\sqrt{|\bm{a}|_m}}\tilde{\psi}\left(\frac{\bm{t}-\bm{b}}{\bm{a}}\right)}d\bm{t}\\
&=&\left\langle f(\bm{t}),\overline{c(\bm\alpha_\lambda)}e_{\bm{\alpha},\lambda^2}(\bm{b})e_{\bm{\alpha},-\lambda^2}(\bm{t}){\frac{1}{\sqrt{|\bm{a}|_m}}\tilde{\psi}}\left(\frac{\bm{t}-\bm{b}}{\bm{a}}\right)\right\rangle\\
&=&\left\langle f(\bm{t}),\mathbf{\Psi}_{\bm{\alpha},\lambda,\bm{a},\bm{b}}(\bm{t})\right\rangle,
\end{eqnarray*}
where
${\mathbf{\Psi}}_{\bm{\alpha},\lambda,\bm{a},\bm{b}}(\bm{t})=\overline{c(\bm\alpha_\lambda)}e_{\bm{\alpha},\lambda^2}(\bm{b})e_{\bm{\alpha},-\lambda^2}(\bm{t}){\frac{1}{\sqrt{|\bm{a}|_m}}\tilde{\psi}}\left(\frac{\bm{t}-\bm{b}}{\bm{a}}\right).
$\\
We now prove some properties of the proposed MFrWT.
\begin{theorem}
Let $\psi,\phi\in L^2(\mathbb{R}^{\N})$ be two wavelets and for any two functions $f,g \in L^2(\mathbb{R}^{\N}).$ Also let $\sigma>0$ and $r,s\in \mathbb{C},$ then
\begin{enumerate}[label=(\roman*)]
\item {Linearity}:\label{P1PropMFrWT1}
$W^{\bm\alpha,\lambda}_{\psi}(r{f}+s{g})=r(W_{\psi}^{\bm\alpha,\lambda}{f})+s(W_{\psi}^{\bm\alpha,\lambda}{g})$.
This property shows that the MFrWT upholds the superposition principle, which is useful for multicomponent signal analysis.
\item{Anti-linearity}:\label{P1PropMFrWT2}
 $W_{\left(r{\psi}+s{\phi}\right)}^{\bm\alpha,\lambda}{f}=\overline{r}W_{{\psi}}{f}+\overline{s}W_{\phi}{f}.$
\item {Dilation}:\label{P1PropMFrWT3}
$(W_{\psi}^{\bm\alpha,\lambda}D_\sigma f)(\bm{a},\bm{b})=\displaystyle\frac{C'}{\sigma^{\N}}W_{\psi}^{\bm\alpha,\frac{\lambda}{\sigma}}f(\sigma\bm{a},\sigma\bm{b}),$ where $D_\sigma$ is a dilation operator, which is given by $D_\sigma f(\bm{t})=f(\sigma \bm{t})$ and $C'=\sigma^{\frac{\N}{2}}\overline{\left(\frac{c(\alpha_{\lambda})}{c\left(\alpha_{\frac{\lambda}{\sigma}}\right)}\right)},$
\item {Conjugacy}: \label{P1PropMFrWT4}
$(W_{\psi}^{\bm\alpha,\lambda}\overline{f})(\bm{a},\bm{b})=\overline{W_{\overline{\psi}}^{-\bm\alpha,\lambda}f(\bm{a},\bm{b})}.$
 \item {Parity}:\label{P1PropMFrWT5}
$(W_{\psi}^{\bm\alpha,\lambda}Pf)(\bm{a},\bm{b})=W_{\psi}^{\bm\alpha,\lambda}f(-\bm{a},-\bm{b}),$
Here, $P$ is the parity operator, which is defined by $Pf(\mathbf{x})=f(-\mathbf{x}).$

 \item{Translation}:\label{P1PropMFrWT6}
 $(W_\psi^{\bm\alpha,\lambda}\tau_{\mathbf{y}}f)(\bm{a},\bm{b})=e_{\bm{\alpha},-\lambda^2}(\bm{b})e_{\bm{\alpha},\lambda^2}(\bm{b}-\mathbf{y})e_{\bm{\alpha},\lambda^2}(\mathbf{y})(W_\psi^{\bm\alpha,\lambda}\check{f})(\bm{a},\bm{b}-\mathbf{y}),$ where  fractional translation operator $\tau_{\mathbf{y}}$ is given by $\tau_{\mathbf{y}}f(\mathbf{x})=f(\mathbf{x}-\mathbf{y})$ and $\displaystyle \check{f}(\bm{t})=f(\bm{t})e^{i\lambda^2\sum_{k=1}^{\N}a(\alpha_k)2{t_{k}}y_k}.$
\end{enumerate}
\end{theorem}
\begin{proof}
The proof of \ref{P1PropMFrWT1} and \ref{P1PropMFrWT2} are straight forward and thus can be omitted.\\

\ref{P1PropMFrWT3} Using the definition of $D_{\sigma}f,$ we have
\begin{eqnarray}\label{P1eqn3}\notag
(W_{\psi}^{\bm\alpha,\lambda}D_{\sigma}f)(\bm{a},\bm{b})=\int_{\mathbb{R}^{\N}}(D_{\sigma}f)(\bm{t})\overline{{\mathbf{\Psi}}_{\bm{\alpha},\lambda,\bm{a},\bm{b}}(\bm{t})}d\bm{t}\\ 
=\displaystyle{\int_{\mathbb{R}^{\N}}}f(\sigma\bm{t})\overline{{\mathbf{\Psi}}_{\bm{\alpha},\lambda,\bm{a},\bm{b}}(\bm{t})}d\bm{t},
\end{eqnarray}
on putting $\sigma\bm{t}=\bm{u}$ and
$d\bm{t}=\displaystyle\frac{d\bm{u}}{\sigma}$ in equation (\ref{P1eqn3}), we get
\begin{eqnarray}\label{P1eqn4}
(W_{\psi}^{\bm\alpha,\lambda}f_{\sigma})(\bm{a},\bm{b})=\frac{1}{\sigma^{\N}}\displaystyle{\int_{\mathbb{R}^{\N}}}f(\bm{u})\overline{{\mathbf{\Psi}}_{\bm{\alpha},\lambda,\bm{a},\bm{b}}\left(\frac{\bm{u}}{\sigma}\right)}d\bm{u}.
\end{eqnarray}
Now, 
\begin{eqnarray*} 
{\mathbf{\Psi}}_{\bm{\alpha},\lambda,\bm{a},\bm{b}}(\bm{t})&=&\overline{c(\bm\alpha_{\lambda})} e_{\bm{\alpha},\lambda^2}(\bm{b})e_{\bm{\alpha},-\lambda^2}(\bm{t})\frac{1}{\sqrt{|\bm{a}|_m}}\tilde{\psi}\left(\frac{\bm{t}-\bm{b}}{\bm{a}}\right)\\ 
 \Rightarrow {\mathbf{\Psi}}_{\bm{\alpha},\lambda,\bm{a},\bm{b}}\left(\frac{\bm{u}}{\sigma}\right)&=&\overline{c(\bm\alpha_{\lambda})}e_{\bm{\alpha},\lambda^2}(\bm{b})e_{\bm{\alpha},-\lambda^2}\left(\frac{\bm{u}}{\sigma}\right)\frac{1}{\sqrt{|\bm{a}|_m}}\tilde{\psi}\left(\frac{(\frac{\bm{u}}{\sigma})-\bm{b}}{\bm{a}}\right)\\ 
&=&\overline{c(\bm\alpha_{\lambda})}e_{\bm{\alpha},\frac{\lambda^2}{\sigma^2}}(\sigma\bm{b})e_{\bm{\alpha},\frac{-\lambda^2}{\sigma^2}}(\bm{u})\frac{1}{\sqrt{|\bm{a}|_m}}\tilde{\psi}\left(\frac{\bm{u}-\sigma\bm{b}}{\sigma{\bm{a}}}\right).\\
&=&\sigma^{\frac{{\N}}{2}}\overline{c(\bm\alpha_{\lambda})}e_{\bm{\alpha},\frac{\lambda^2}{\sigma^2}}(\sigma\bm{b})e_{\bm{\alpha},\frac{-\lambda^2}{\sigma^2}}(\bm{u})\frac{1}{\sqrt{|\sigma\bm{a}|_m}}\tilde{\psi}\left(\frac{\bm{u}-\sigma\bm{b}}{\sigma{\bm{a}}}\right).
\end{eqnarray*}
Therefore,
\begin{eqnarray}\label{P1eqn5}
\mathbf{\Psi}_{\bm{\alpha},\lambda,\bm{a},\bm{b}}\left(\frac{\bm{u}}{\sigma}\right)=C'{\mathbf{\Psi}}_{\bm{\alpha},\frac{\lambda}{\sigma},\sigma\bm{a},\sigma\bm{b}}(\bm{u}),
\end{eqnarray}
 where $C'=\sigma^{\frac{\N}{2}}\overline{\left(\frac{c(\alpha_{\lambda})}{c\left(\alpha_{\frac{\lambda}{\sigma}}\right)}\right)}.$
We now get the required result by putting equation (\ref{P1eqn5}) into equation (\ref{P1eqn4}).\\

\ref{P1PropMFrWT4}
\begin{eqnarray*}
(W_{\psi}^{\bm\alpha,\lambda}\overline{f})(\bm{a},\bm{b})&=&\left\langle \overline{f}(\bm{t}),\mathbf{\Psi}_{\bm{\alpha},\lambda,\bm{a},\bm{b}}(\bm{t})\right\rangle_{L^2(\mathbb{R}^{\N})}\\
&=&\left\langle \overline{f}(\bm{t}),{c((\boldsymbol-\alpha)_\lambda)}e_{\bm{\alpha},\lambda^2}(\bm{b})e_{\bm{\alpha},-\lambda^2}(\bm{t}){\frac{1}{\sqrt{|\bm{a}|_m}}\tilde{\psi}}\left(\frac{\bm{t}-\bm{b}}{\bm{a}}\right)\right\rangle_{L^2(\mathbb{R}^{\N})}\\
&=&\overline{\left\langle f(\bm{t}),\overline{c((\boldsymbol-\alpha)_\lambda)}e_{\boldsymbol{-\alpha},\lambda^2}(\bm{b})e_{\boldsymbol{-\alpha},-\lambda^2}(\bm{t}){\frac{1}{\sqrt{|\bm{a}|_m}}\overline{\tilde{\psi}\left(\frac{\bm{t}-\bm{b}}{\bm{a}}\right)}}\right\rangle_{L^2(\mathbb{R}^{\N})}}\\
&=&\overline{\left\langle f(\bm{t}),\overline{c((\boldsymbol-\alpha)_\lambda)}e_{\boldsymbol{-\alpha},\lambda^2}(\bm{b})e_{\boldsymbol{-\alpha},-\lambda^2}(\bm{t})\frac{1}{\sqrt{|\bm{a}|_m}}\overline{({\psi e_{\alpha,\lambda^2})}\left(\frac{\bm{t}-\bm{b}}{\bm{a}}\right)}\right\rangle_{L^2(\mathbb{R}^{\N})}}\\
&=&\overline{\left\langle f(\bm{t}),\overline{c((\boldsymbol-\alpha)_\lambda)}e_{\boldsymbol{-\alpha},\lambda^2}(\bm{b})e_{\boldsymbol{-\alpha},-\lambda^2}(\bm{t})\frac{1}{\sqrt{|\bm{a}|_m}}(\overline{\psi} e_{-\alpha,\lambda^2})\left(\frac{\bm{t}-\bm{b}}{\bm{a}}\right)\right\rangle_{L^2(\mathbb{R}^{\N})}}\\
&=&\overline{W_{\overline{\psi}}^{-\bm\alpha,\lambda}f(\bm{a},\bm{b})}.
\end{eqnarray*}

\ref{P1PropMFrWT5}  Using the definition of $Pf,$ we have
\begin{eqnarray}\label{P1eqn5a}     \notag
(W_{\psi}^{\bm\alpha,\lambda}Pf)(\bm{a},\bm{b})=\displaystyle\int_{\mathbb{R}^{\N}}Pf(\bm{t})\overline{\mathbf{\Psi}_{\bm{\alpha},\lambda,\bm{a},\bm{b}}(\bm{t})}d\bm{t} \\ 
=\displaystyle{\int_{\mathbb{R}^{\N}}}f(-\bm{t})\overline{{\mathbf{\Psi}}_{\bm{\alpha},\lambda,\bm{a},\bm{b}}(\bm{t})}d\bm{t}.
\end{eqnarray}

Putting $-\bm{t}=\mathbf{x}$ in equation (\ref{P1eqn5a}), we get
\begin{eqnarray}\label{P1eqn5b} 
(W_{\psi}^{\bm\alpha,\lambda}Pf)(\bm{a},\bm{b})=\displaystyle{\int_{\mathbb{R}^{\N}}}f(\mathbf{x})\overline{\mathbf{\Psi}_{\bm{\alpha},\lambda,\bm{a},\bm{b}}(-\mathbf{x}})d\mathbf{x}.
\end{eqnarray}
Now,
\begin{eqnarray}\label{P1eqn5c} \notag
\mathbf{\Psi}_{\bm{\alpha},\lambda,\bm{a},\bm{b}}(-\bm{t})&=&\overline{c(\bm\alpha_{\lambda})}e_{\bm{\alpha},\lambda^2}(\bm{b})e_{\bm{\alpha},-\lambda^2}(-\bm{t})\frac{1}{\sqrt{|\bm{a}|_m}}\tilde{\psi}\left(\frac{-\bm{t}-\bm{b}}{\bm{a}}\right)\\ \notag
&=&\overline{c(\bm\alpha_{\lambda})}e_{\bm{\alpha},\lambda^2}(-\bm{b})e_{\bm{\alpha},-\lambda^2}(\bm{t})\frac{1}{\sqrt{|\bm{a}|_m}}\tilde{\psi}\left(\frac{\bm{t}-(-\bm{b})}{-\bm{a}}\right)\\
&=&{\mathbf{\Psi}}_{\bm{\alpha},\lambda,-\bm{a},-\bm{b}}(\bm{t}).
\end{eqnarray}
We get the required result by putting equation (9) into equation (8).\\

\ref{P1PropMFrWT6} Using the definition of $\tau_{\mathbf{y}}f,$ we get
\begin{eqnarray}\label{P1eqn5d} \notag
(W_\psi^{\bm\alpha,\lambda}\tau_{\mathbf{y}}f)(\bm{a},\bm{b})&=&\displaystyle\int_{\mathbb{R}^{\N}}\tau_{\mathbf{y}}f(\mathbf{x})\overline{\mathbf{\Psi}_{\bm{\alpha},\lambda,\bm{a},\bm{b}}(\mathbf{x})}d\mathbf{x} \\ 
&=&\displaystyle\int_{\mathbb{R}^{\N}}f(\mathbf{x}-\mathbf{y})\overline{\mathbf{\Psi}_{\bm{\alpha},\lambda,\bm{a},\bm{b}}(\mathbf{x})}d\mathbf{x},
\end{eqnarray}
 and by putting $\mathbf{x}-\mathbf{y}=\bm{t}, \mathbf{x}=\bm{t}+\mathbf{y}, d\mathbf{x}=d\bm{t}$ in equation (\ref{P1eqn5d}), we get
\begin{eqnarray}\label{P1eqn5e}
 (W_\psi^{\bm\alpha,\lambda}\tau_{\mathbf{y}}f)(\bm{a},\bm{b})=\displaystyle\int_{\mathbb{R}^{\N}}f(\bm{t})\overline{\mathbf{\Psi}_{\bm\alpha,\lambda,\bm{a},\bm{b}}(\bm{t}+\mathbf{y})}d\bm{t}.
\end{eqnarray}
Now,
\begin{align} \label{P1eqn5f} \notag
\mathbf{\Psi}_{\bm\alpha,\lambda,\bm{a},\bm{b}}(\bm{t}+\mathbf{y}) = &  \overline{c(\bm\alpha_{\lambda})}e_{\bm{\alpha},\lambda^2}(\bm{b})e_{\bm{\alpha},-\lambda^2}(\bm{t}+\mathbf{y}){\frac{1}{\sqrt{|\bm{a}|_m}}\tilde{\psi}}\left(\frac{\bm{t}+\mathbf{y}-\bm{b}}{\bm{a}}\right)\\ \notag
 = &\overline{c(\bm\alpha_{\lambda})}e_{\bm{\alpha},\lambda^2}(\bm{b})e_{\bm{\alpha},-\lambda^2}(\bm{t}+\mathbf{y}){\frac{1}{\sqrt{|\bm{a}|_m}}\tilde{\psi}}\left(\frac{\bm{t}-(\bm{b}-\mathbf{y})}{\bm{a}}\right)\\ \notag
= &e_{\bm{\alpha},\lambda^2}(\bm{b})e_{\bm{\alpha},-\lambda^2}(\bm{b}-\mathbf{y})e_{\bm{\alpha},-\lambda^2}(\mathbf{y})e^{-i\lambda^2\sum_{k=1}^{\N}a(\alpha_k)2{t_{k}}y_k}\Big[\overline{c(\alpha_{\lambda})}e_{\bm{\alpha},\lambda^2}(\bm{b}-\mathbf{y})e_{\bm{\alpha},-\lambda^2}(\bm{t})\\\notag
&~~~{\frac{1}{\sqrt{|\bm{a}|_m}}}\tilde{\psi}\left(\frac{\bm{t}-(\bm{b}-\mathbf{y})}{\bm{a}}\right)\Big]\\ 
= &e_{\bm{\alpha},\lambda^2}(\bm{b})e_{\bm{\alpha},-\lambda^2}(\bm{b}-\mathbf{y})e_{\bm{\alpha},-\lambda^2}(\mathbf{y})e^{-i\lambda^2\sum_{k=1}^{\N}a(\alpha_k)2{t_{k}}y_k}{\mathbf{\Psi}}_{\bm\alpha,\lambda,\bm{a},\bm{b}-\mathbf{y}}(\bm{t}).
\end{align}

We get the required result by putting equation(\ref{P1eqn5f}) into equation (\ref{P1eqn5e}).
\end{proof}
Now we will derive some lemmas that will help us to demonstrate the inner product relation for MFrWT . 
\begin{lemma}
 (FrFT of $\mathbf{\Psi}_{\bm{\alpha},\lambda,\bm{a},\bm{b}}(\bm{t})$) Suppose a non-zero function  $\psi\in L^2(\mathbb{R}^{\N})$ is given, then 
\begin{eqnarray}
(\mathfrak{F}_{\bm{\alpha},\lambda}{\mathbf{\Psi}}_{\bm{\alpha},\lambda,\bm{a},\bm{b}})(\bm{\xi})&=&\overline{c(\bm\alpha_\lambda)}\sqrt{|\bm{a}|_m} e_{\bm{\alpha},\lambda^2}(\bm{b})e_{\alpha,\lambda^2}(\bm{\xi})e_{\bm{\alpha},-\lambda^2}(\bm{a}\bm{\xi})e^{-i\lambda^2\sum_{k=1}^{\N}b_{k}\xi_{k}\csc\alpha_{k}}(\mathfrak{F}_{\bm{\alpha},\lambda}\psi)(\bm{a}\bm{\xi}).
\end{eqnarray}
\end{lemma}

\begin{proof}
For any $ \bm{\xi}\in \mathbb{R}^{\N},$
\begin{eqnarray} \label{P1eqn6}\notag
(\mathfrak{F}_{\bm{\alpha},\lambda}{\mathbf{\Psi}}_{\bm{\alpha},\lambda,\bm{a},\bm{b}})(\bm{\xi}) &=& \displaystyle\int_{\mathbb{R}^{\N}}\mathbf{\Psi}_{\bm{\alpha},\lambda,\bm{a},\bm{b}}(\bm{t})K_{\bm{\alpha},\lambda}\left(\bm{t},\bm{\xi}\right) d\bm{t}\\ \notag
&=&\overline{c(\bm\alpha_\lambda)}e_{\bm{\alpha},\lambda^2}(\bm{b})\displaystyle\int_{\mathbb{R}^{\N}}e_{\bm{\alpha},-\lambda^2}(\bm{t})\frac{1}{\sqrt{|\bm{a}|_m}}\tilde{\psi}\left(\frac{\bm{t}-\bm{b}}{\bm{a}}\right)\frac{c(\bm\alpha_\lambda)}{(\sqrt{2\pi})^{\N}}e_{\bm{\alpha},\lambda^2}(\bm{t})e_{\bm{\alpha},\lambda^2}(\bm{\xi})e^{-i\lambda^2\sum_{k=1}^{\N}t_{k}\xi_{k}\csc\alpha_{k}}d\bm{t}\\
&=&\frac{1}{\sqrt{|\bm{a}|_m}}\frac{|c({\bm\alpha_\lambda})|^2}{(\sqrt{2\pi})^{\N}}e_{\bm{\alpha},\lambda^2}(\bm{b})\displaystyle\int_{\mathbb{R}^{\N}}\tilde{\psi}\left(\frac{\bm{t}-\bm{b}}{\bm{a}}\right)e_{\bm{\alpha},\lambda^2}(\bm{\xi})e^{-i\lambda^2\sum_{k=1}^{\N}t_{k}\xi_{k}\csc\alpha_{k}}d\bm{t}.
\end{eqnarray}
Putting $\bm{t}=\bm{a}\bm{u}+\bm{b}$ and $d\bm{t}=|\bm{a}|_md\bm{u}$ in
equation (\ref{P1eqn6}), we've
\begin{align*}
(\mathfrak{F}_{\bm{\alpha},\lambda}&{\mathbf{\Psi}}_{\bm{\alpha},\lambda,\bm{a},\bm{b}})(\bm{\xi})\\&=\frac{|c({\bm\alpha_\lambda})|^2}{(\sqrt{2\pi})^{\N}}\sqrt{|\bm{a}|_m}e_{\bm{\alpha},\lambda^2}(\bm{b})e_{\bm{\alpha},\lambda^2}(\bm{\xi})e_{\bm{\alpha},-\lambda^2}(\bm{a}\bm{\xi})e^{-i\lambda^2\sum_{k=1}^{\N}b_{k}\xi_{k}\csc\alpha_{k}}\int_{\mathbb{R}^{\N}}\tilde{\psi}(\bm{u})e_{\bm{\alpha},\lambda^2}(\bm{a}\bm{\xi})e^{-i\lambda^2\sum_{k=1}^{\N}a_{k}u_{k}\xi_{k}\csc\alpha_{k}}d\bm{u}\\
&=\overline{c(\bm\alpha_\lambda)}\sqrt{|\bm{a}|_m}e_{\bm{\alpha},\lambda^2}(\bm{b})e_{\bm{\alpha},\lambda^2}(\bm{\xi})e_{\bm{\alpha},-\lambda^2}(\bm{a}\bm{\xi})e^{-i\lambda^2\sum_{k=1}^{\N}b_{k}\xi_{k}\csc\alpha_{k}}\int_{\mathbb{R}^{\N}}\psi(\bm{u})K_{\bm{\alpha},\lambda}(\bm{u},\bm{a}\bm{\xi})d\bm{u}\\
&=\overline{c(\bm\alpha_\lambda)}\sqrt{|\bm{a}|_m} e_{\bm{\alpha},\lambda^2}(\bm{b})e_{\bm{\alpha},\lambda^2}(\bm{\xi})e_{\bm{\alpha},-\lambda^2}(\bm{a}\bm{\xi})e^{-i\lambda^2\sum_{k=1}^{\N}b_{k}\xi_{k}\csc\alpha_{k}}(\mathfrak{F}_{\bm{\alpha},\lambda}\psi)(\bm{a}\bm{\xi}).
\end{align*}
This completes the proof.
\end{proof}

\begin{lemma}\label{P1lemma1}
(FrFT of $W_{\psi}^{\bm\alpha,\lambda}{f}(\bm{a},\bm{b})$) Suppose $\psi$ be an admissible wavelet and for any arbitrary function $f\in L^2(\mathbb{R}^{\N})$, then\\
\begin{eqnarray*}
\mathfrak{F}_{\bm{\alpha},\lambda}(W_{\psi}^{\bm\alpha,\lambda}{f})(\bm{a},\bm{b})(\bm{\xi})=\sqrt{|\bm{a}|_m}(\sqrt{2\pi})^{\N} \frac{c(\bm\alpha_\lambda)}{\overline{c(\bm\alpha_\lambda)}}F(\bm{\xi}),\\
\end{eqnarray*}
where $F(\bm{\xi})=e_{\bm{\alpha},\lambda^2}(\bm{a}\bm{\xi})(\mathfrak{F}_{\bm{\alpha},\lambda}f)(\bm{\xi})\overline{(\mathfrak{F}_{\bm{\alpha},\lambda}\psi)(\bm{a}\bm{\xi})}.$
\end{lemma}
\begin{proof}
Here, we know 
\begin{align*}
(W_{\psi}^{\bm\alpha,\lambda}{f})(\bm{a},\bm{b})&=\left\langle f(\bm{t}),\mathbf{\Psi}_{\bm{\alpha},\lambda,\bm{a},\bm{b}}(\bm{t})\right\rangle\\ 
\end{align*}
Therefore,
\begin{align}\label{P1eqn7} \notag
(W_{\psi}^{\bm\alpha,\lambda}{f})(\bm{a},\bm{b})
&=\left\langle (\mathfrak{F}_{\bm{\alpha},\lambda}f)(\bm{\xi}),(\mathfrak{F}_{\bm{\alpha},\lambda}{\mathbf{\Psi}}_{\bm{\alpha},\lambda,\bm{a},\bm{b}})(\bm{\xi}))\right\rangle\\ \notag
&=\displaystyle\int_{\mathbb{R}^{\N}} (\mathfrak{F}_{\bm{\alpha},\lambda}f)(\bm{\xi})\overline{(\mathfrak{F}_{\bm{\alpha},\lambda}\mathbf{\Psi}_{\bm\alpha,\lambda,\bm{a},\bm{b}})(\bm{\xi})}d\bm{\xi}\\ 
&=\displaystyle\int_{\mathbb{R}^{\N}}(\mathfrak{F}_{\bm{\alpha},\lambda}f)(\bm{\xi}){c(\bm\alpha_\lambda)}\sqrt{|\bm{a}|_m} e_{\bm{\alpha},-\lambda^2}(\bm{b})e_{\bm{\alpha},-\lambda^2}(\bm{\xi})e_{\bm{\alpha},\lambda^2}(\bm{a}\bm{\xi})e^{i\lambda^2\sum_{k=1}^{\N}b_{k}\xi_{k}\csc\alpha_{k}}\overline{(\mathfrak{F}_{\bm{\alpha},\lambda}\psi)(\bm{a}\bm{\xi})}d\bm{\xi}.
\end{align}

From equation (\ref{P1eqn7}), we have
$$(W_{\psi}^{\bm\alpha,\lambda}{f})(\bm{a},\bm{b})= {c(\bm\alpha_\lambda)}\sqrt{|\bm{a}|_m} \int_{\mathbb{R}^{\N}}e_{\bm{\alpha},-\lambda^2}(\bm{b})e_{\bm{\alpha},-\lambda^2}(\bm{\xi})e^{i\lambda^2\sum_{k=1}^{\N}b_{k}\xi_{k}\csc\alpha_{k}}F(\bm{\xi})d\bm{\xi},$$
where
\begin{align}\label{P1eqn8} F(\bm{\xi})=e_{\bm{\alpha},\lambda^2}(\bm{a}\bm{\xi})(\mathfrak{F}_{\bm{\alpha},\lambda})(\bm{\xi})\overline{(\mathfrak{F}_{\bm{\alpha},\lambda}\psi)(\bm{a} \bm{\xi})}.
\end{align}
Thus, we have
\begin{align} \notag
(W_{\psi}^{\bm\alpha,\lambda}{f})(\bm{a},\bm{b})
%&= {c(\bm\alpha_\lambda)}\sqrt{|\bm{a}|_m} \int_{\mathbb{R}^{\N}}e_{\bm{\alpha},-\lambda^2}(\bm{b})e_{\bm{\alpha},-\lambda^2}(\bm{\xi})e^{i\lambda^2\sum_{k=1}^{\N}b_{k}\xi_{k}\csc\alpha_{k}}F(\bm{\xi})d\bm{\xi}\\ \notag
&=\sqrt{|\bm{a}|_m}(\sqrt{2\pi})^{\N} \frac{c(\bm\alpha_\lambda)}{\overline{c(\bm\alpha_\lambda)}}\int_{\mathbb{R}^{\N}}K_{-\bm{\alpha},\lambda}\left(\bm{b},\bm{\xi}\right)F(\bm{\xi})d\bm{\xi}\\ \notag
&=\sqrt{|\bm{a}|_m}(\sqrt{2\pi})^{\N}\frac{c(\bm\alpha_\lambda)}{\overline{c(\bm\alpha_\lambda)}} \mathfrak{F}_{-\bm{\alpha},\lambda}(F(\bm{\xi}))(\bm{b}),
\end{align}
which implies that
\begin{eqnarray}{\label{p1eqn9}}
\mathfrak{F}_{\bm{\alpha},\lambda} (W_{\psi}^{\bm\alpha,\lambda}{f})(\bm{a},\bm{b})(\bm{\xi})&=&\sqrt{|\bm{a}|_m}(\sqrt{2\pi})^{\N}\frac{c(\bm\alpha_\lambda)}{\overline{c(\bm\alpha_\lambda)}} F(\bm{\xi}).
\end{eqnarray}
We get the required result by putting equation (\ref{P1eqn8}) into equation (\ref{p1eqn9}).
\end{proof}
\begin{theorem}(Inner product relation for MFrWT){\label{P1innTh}}
Let $W_{\psi}^{\bm\alpha,\lambda}{f}$ and $W_{\psi}^{\bm\alpha,\lambda}{g}$  represent the MFrWT of functions $f$ and $g$ with respect to admissible wavelet $\psi$ respectively and then 
\begin{eqnarray}\label{P1eqn10}
\displaystyle\int_{\mathbb{R}^{\N}_{0}}\displaystyle\int_{\mathbb{R}^{\N}}(W_{\psi}^{\bm\alpha,\lambda}{f})(\bm{a},\bm{b})\overline{(W_{\psi}^{\bm\alpha,\lambda}{g})(\bm{a},\bm{b})}\frac{d\bm{a}d\bm{b}}{|\bm{a}|_m^2}=C_{\bm{\alpha},\lambda}\langle f,g\rangle_{L^2(\mathbb{R}^{\N})},
\end{eqnarray}
where
\begin{eqnarray}\label{P1eqn11}
C_{\bm{\alpha},\lambda}={(2\pi)^{\N}}\displaystyle\int_{\mathbb{R}^{\N}_{0}}|(\mathfrak{F}_{\bm{\alpha},\lambda}\psi)(\bm{u})|^2\frac{d\bm{u}}{|\bm{u}|_m}.
\end{eqnarray}
\end{theorem}
\begin{proof}
Treating $(W_{\psi}^{\bm\alpha,\lambda}{f})(\bm{a},\bm{b})$, $(W_{\psi}^{\bm\alpha,\lambda}{g})(\bm{a},\bm{b})$ as $L^2(\mathbb{R}^{\N})$ are functions of some variable $\bm{b}$ and using the Parseval’s theorem
for the MFrFT  (\cite{kamalakkannan2020multidimensional}, Theorem 3.3) we have
\begin{align}\label{P1eqn12}
\displaystyle\int_{\mathbb{R}^{\N}_{0}}\displaystyle\int_{\mathbb{R}^{\N}}(W_{\psi}^{\bm\alpha,\lambda}{f})(\bm{a},\bm{b})\overline{(W_{\psi}^{\bm\alpha,\lambda}{g})(\bm{a},\bm{b})}\frac{d\bm{a}d\bm{b}}{|\bm{a}|_m^2}=\displaystyle\int_{\mathbb{R}^{\N}_{0}}\left(\displaystyle\int_{\mathbb{R}^{\N}}\mathfrak{F}_{\bm{\alpha},\lambda}( (W_{\psi}^{\bm\alpha,\lambda}{f})(\bm{a},\bm{b}))(\bm{\xi})\overline{\mathfrak{F}_{\bm{\alpha},\lambda}(( W_{\psi}^{\bm\alpha,\lambda}{g})(\bm{a},\bm{b}))(\bm{\xi})}d\bm{\xi}\right)\frac{d\bm{a}}{|\bm{a}|_m^2}.
\end{align}
Now, using equation (\ref{p1eqn9}) in equation (\ref{P1eqn12}), we get
\begin{align*}
&\displaystyle\int_{\mathbb{R}^{\N}_{0}}\displaystyle\int_{\mathbb{R}^{\N}}(W_{\psi}^{\bm\alpha,\lambda}{f})(\bm{a},\bm{b})\overline{(W_{\psi}^{\bm\alpha,\lambda}{g})(\bm{a},\bm{b})}\frac{d\bm{a}d\bm{b}}{|\bm{a}|_m^2}\\&=\displaystyle\int_{\mathbb{R}^{\N}_{0}}\left(\int_{\mathbb{R}^{\N}}\sqrt{|\bm{a}|_m}{(\sqrt{2\pi})^{\N}}e_{\bm{\alpha},\lambda^2}(\bm{a}\bm{\xi})\overline{(\mathfrak{F}_{\bm{\alpha},\lambda}\psi)(\bm{a} \bm{\xi})}(\mathfrak{F}_{\bm{\alpha},\lambda} f)(\bm{\xi})\overline{\sqrt{|\bm{a}}|_m(\sqrt{2\pi})^Ne_{\bm{\alpha},\lambda^2}(\bm{a}\bm{\xi})\overline{(\mathfrak{F}_{\bm{\alpha},\lambda}\psi)(\bm{a} \bm{\xi})}(\mathfrak{F}_{\bm{\alpha},\lambda} g)(\bm{\xi})}d\bm{\xi}\right)\frac{d\bm{a}}{|\bm{a}|_m^2}\\
&=(2\pi)^{\N}\displaystyle\int_{\mathbb{R}^{\N}_{0}}\displaystyle\int_{\mathbb{R}^{\N}}\overline{(\mathfrak{F}_{\bm{\alpha},\lambda}\psi)(\bm{a} \bm{\xi})}(\mathfrak{F}_{\bm{\alpha},\lambda}\psi)(\bm{a} \bm{\xi})(\mathfrak{F}_{\bm{\alpha},\lambda} f)(\bm{\xi})\overline{(\mathfrak{F}_{\bm{\alpha},\lambda} g)(\bm{\xi})}d\bm{\xi}\frac{d\bm{a}}{|\bm{a}|_m}\\
&=\int_{\mathbb{R}^{\N}}(\mathfrak{F}_{\bm{\alpha},\lambda}f)(\bm{\xi})\overline{(\mathfrak{F}_{\bm{\alpha},\lambda}g)(\bm{\xi})}\left(\displaystyle\int_{\mathbb{R}^{\N}_{0}}(2\pi)^{\N}(\mathfrak{F}_{\bm{\alpha},\lambda}\psi)(\bm{a} \bm{\xi})\overline{(\mathfrak{F}_{\bm{\alpha},\lambda}\psi)(\bm{a} \bm{\xi})}\frac{d\bm{a}}{|\bm{a}|_m}\right)d\bm{\xi}.
\end{align*}
Thus we have
\begin{align*}
\displaystyle\int_{\mathbb{R}^{\N}_{0}}\displaystyle\int_{\mathbb{R}^{\N}}(W_{\psi}^{\bm\alpha,\lambda}{f})(\bm{a},\bm{b})\overline{(W_{\psi}^{\bm\alpha,\lambda}{g})(\bm{a},\bm{b})}\frac{d\bm{a}d\bm{b}}{|\bm{a}|_m^2}&=\displaystyle\int_{\mathbb{R}^{\N}}C_{\bm{\alpha},\lambda}(\mathfrak{F}_{\bm{\alpha},\lambda}f)(\bm{\xi})\overline{(\mathfrak{F}_{\bm{\alpha},\lambda}g)(\bm{\xi})}d\bm{\xi}\\
&=C_{\bm{\alpha},\lambda}\left\langle(\mathfrak{F}_{\bm{\alpha},\lambda}f)(\bm{\xi}),(\mathfrak{F}_{\bm{\alpha},\lambda}g)(\bm{\xi})\right\rangle_{L^2(\mathbb{R}^{\N})}. 
%&= C_{\bm{\alpha},\lambda}\langle f,g\rangle_{L^2(\mathbb{R}^{\N})}.
\end{align*}   
Again, applying Parseval's inequality (\cite{kamalakkannan2020multidimensional}, Theorem 3.3), we get the required result. Fubini's theorem has been used to justify changing the order of integration.
\end{proof}
\begin{corollary}\label{corofinnthe}
Suppose $\psi$ be an admissible wavelet and for any arbitrary function $f\in L^2(\mathbb{R}^{\N})$. Then
\begin{eqnarray*}
\int_{\mathbb{R}^{\N}_{0}}\int_{\mathbb{R}^{\N}}(W_{\psi}^{\bm\alpha,\lambda}{f})(\bm{a},\bm{b})\overline{(W_{\psi}^{\bm\alpha,\lambda}{f})(\bm{a},\bm{b})}\frac{d\bm{a}d\bm{b}}{|\bm{a}|_m^2}=C_{\bm{\alpha},\lambda}\|f\|^2_{L^2(\mathbb{R}^{\N})}.
\end{eqnarray*}
\end{corollary}
\begin{theorem}(Reconstruction formula)
Suppose there exist an arbitrary function $f\in L^2(\mathbb{R}^{\N})$ and an wavelet admissible function $\psi.$  Also let $C_{\bm{\alpha},\lambda}\neq 0 $ as defined in equation (\ref{P1eqn11}). The reconstruction formula for $f$ is then as follows:
\begin{eqnarray*}
f(\bm{t})=\frac{1}{C_{\bm{\alpha},\lambda}}\int_{\mathbb{R}^{\N}_{0}}\int_{\mathbb{R}^{\N}}(W_{\psi}^{\bm\alpha,\lambda}{f})(\bm{a},\bm{b})\mathbf{\Psi}_{\bm{\alpha},\lambda,\bm{a},\bm{b}}(\bm{t})\frac{d\bm{a}d\bm{b}}{|\bm{a}|_m^2}.
\end{eqnarray*}
\end{theorem}
\begin{proof}
We have
\begin{align*}
\left\langle W_{\psi}^{\bm\alpha,\lambda}{f}(\bm{a},\bm{b}),W_{\psi}^{\bm\alpha,\lambda}{g}(\bm{a},\bm{b})\right\rangle_{L^2(\mathbb{R}^{\N}_0\times \mathbb{R}^{\N})}
&=\int_{\mathbb{R}^{\N}_{0}}\int_{\mathbb{R}^{\N}} (W_{\psi}^{\bm\alpha,\lambda}{f})(\bm{a},\bm{b})\overline{(W_{\psi}^{\bm\alpha,\lambda}{g}(\bm{a},\bm{b}))}\frac{d\bm{a}d\bm{b}}{|\bm{a}|_m^2}\\
&=\int_{\mathbb{R}^{\N}_{0}}\int_{\mathbb{R}^{\N}} (W_{\psi}^{\bm\alpha,\lambda}{f})(\bm{a},\bm{b})\overline{\int_{\mathbb{R}^{\N}}g(\bm{t})\overline{\mathbf{\Psi}_{\bm{\alpha},\lambda,\bm{a},\bm{b}}(\bm{t})}d\bm{t}}\frac{d\bm{a}d\bm{b}}{|\bm{a}|_m^2}\\
&=\int_{\mathbb{R}^{\N}_{0}}\int_{\mathbb{R}^{\N}} (W_{\psi}^{\bm\alpha,\lambda}{f})(\bm{a},\bm{b})\int_{\mathbb{R}^{\N}}\overline{g(\bm{t})}\mathbf{\Psi}_{\bm{\alpha},\lambda,\bm{a},\bm{b}}(\bm{t})d\bm{t}\frac{d\bm{a}d\bm{b}}{|\bm{a}|_m^2}\\
&=\int_{\mathbb{R}^{\N}}\overline{g(\bm{t})}\int_{\mathbb{R}^{\N}_{0}}\int_{\mathbb{R}^{\N}}( W_{\psi}^{\bm\alpha,\lambda}{f})(\bm{a},\bm{b})\mathbf{\Psi}_{\bm{\alpha},\lambda,\bm{a},\bm{b}}(\bm{t})\frac{d\bm{a}d\bm{b}}{|\bm{a}|_m^2} d\bm{t}\\
&=\left\langle\int_{\mathbb{R}^{\N}_{0}}\int_{\mathbb{R}^{\N}} (W_{\psi}^{\bm\alpha,\lambda}{f})(\bm{a},\bm{b})\mathbf{\Psi}_{\bm{\alpha},\lambda,\bm{a},\bm{b}}(\bm{t})\frac{d\bm{a}d\bm{b}}{|\bm{a}|_m^2},g(\bm{t})\right\rangle_{L^2(\mathbb{R}^{\N})}.
\end{align*}
Using theorem (\ref{P1innTh})
\begin{eqnarray*}
\langle f,g\rangle_{L^2(\mathbb{R}^{\N})} &=& \frac{1}{C_{\bm\alpha,\lambda}}\left\langle(W_{\psi}^{\bm\alpha,\lambda}{f}),(W_{\psi}^{\bm\alpha,\lambda}{g})\right\rangle_{L^2(\mathbb{R}^{\N}_0\times \mathbb{R}^{\N})}\\
&=&\frac{1}{C_{\bm\alpha,\lambda}}\left\langle\int_{\mathbb{R}^{\N}_{0}}\int_{\mathbb{R}^{\N}} (W_{\psi}^{\bm\alpha,\lambda}{f})(\bm{a},\bm{b})\mathbf{\Psi}_{\bm{\alpha},\lambda,\bm{a},\bm{b}}(\bm{t})\frac{d\bm{a}d\bm{b}}{|\bm{a}|_m^2},g(\bm{t})\right\rangle_{L^2(\mathbb{R}^{\N})}\\
\Rightarrow \quad f(\bm{t})&=&\frac{1}{C_{\bm\alpha,\lambda}}\int_{\mathbb{R}^{\N}_{0}}\int_{\mathbb{R}^{\N}}(W_{\psi}^{\bm\alpha,\lambda}{f})(\bm{a},\bm{b})\mathbf{\Psi}_{\bm{\alpha},\lambda,\bm{a},\bm{b}}(\bm{t})\frac{d\bm{a}d\bm{b}}{|\bm{a}|_m^2}.
\end{eqnarray*}
Hence, the theorem follows.
\end{proof}
Now, we will derive a theorem which gives reproducing kernel for the range of the MFrWT.
\begin{theorem}(Reproducing kernel for the range of MFrWT)
Let $\psi$ be a wavelet. Also let $C_{\alpha,\lambda}\neq 0$ as defined in equation (\ref{P1eqn11}) and $(\bm{a}_0 , \bm{b}_0) \in \mathbb{R}^{\N}\times \mathbb{R}^{\N}_0$. Then, $\F\in L^2\left(\mathbb{R}^{\N}\times \mathbb{R}^{\N}_0, \frac{d\bm{a}d\bm{b}}{|\bm{a}|_m^2}\right)$ be the MFrWT of some $f \in L^2(\mathbb{R}^{\N})$ iff
\begin{eqnarray*}
\F(\bm{a}_0,\bm{b}_0)=\int_{\mathbb{R}^{\N}_0}\int_{\mathbb{R}^{\N}}
\F(\bm{a},\bm{b})K_{\bm{\alpha},\lambda,\mathbf{\Psi}}(\bm{a}_0,\bm{b}_0;\bm{a},\bm{b})\frac{d\bm{a}d\bm{b}}{|\bm{a}|_m^2},
\end{eqnarray*}
where $K_{\bm{\alpha},\lambda,\mathbf{\Psi}}(\bm{a}_0,\bm{b}_0:\bm{a},\bm{b})$ is the reproducing kernel which satisfies,
\begin{eqnarray*}
K_{\bm{\alpha},\lambda,\mathbf{\Psi}}(\bm{a}_0,\bm{b}_0;\bm{a},\bm{b})=\frac{1}{C_{\alpha,\lambda}}\int_{\mathbb{R}^{\N}}\mathbf{\Psi}_{\bm{\alpha},\lambda,\bm{a},\bm{b}}(\bm{t})\overline{\mathbf{\Psi}_{\bm{\alpha},\lambda,\bm{a}_0,\bm{b}_0}(\bm{t})} d\bm{t}.
\end{eqnarray*}
\end{theorem}
\begin{proof}
Suppose for an arbitrary function $f \in L^2(\mathbb{R}^{\N})$ be such that $W_{\psi}^{\bm\alpha,\lambda}{f}=\F$, then we can write,
\begin{eqnarray*}
\F(\bm{a}_0, \bm{b}_0)&=&(W_{\psi}^{\bm\alpha,\lambda}{f})(\bm{a}_0, \bm{b}_0)\\
&=&\int_{\mathbb{R}^{\N}}f(\bm{t})\overline{\mathbf{\Psi}_{\bm{\alpha},\lambda,\bm{a}_0,\bm{b}_0}(\bm{t})} d\bm{t}\\
&=&\frac{1}{C_{\bm\alpha,\lambda}}\int_{\mathbb{R}^{\N}}\int_{\mathbb{R}^{\N}_{0}}\int_{\mathbb{R}^{\N}}(W_{\psi}^{\bm\alpha,\lambda}{f})(\bm{a},\bm{b})\mathbf{\Psi}_{\bm{\alpha},\lambda,\bm{a},\bm{b}}(\bm{t})\frac{d\bm{a}d\bm{b}}{|\bm{a}|_m^2}\overline{\mathbf{\Psi}_{\bm{\alpha},\lambda,\bm{a}_0,\bm{b}_0}(\bm{t})} d\bm{t}\\
&=&\int_{\mathbb{R}^{\N}_{0}}\int_{\mathbb{R}^{\N}}(W_{\psi}^{\bm\alpha,\lambda}{f})(\bm{a},\bm{b})\left(\frac{1}{C_{\bm\alpha,\lambda}}\int_{\mathbb{R}^{\N}}\mathbf{\Psi}_{\bm{\alpha},\lambda,\bm{a},\bm{b}}(\bm{t})\overline{\mathbf{\Psi}_{\bm{\alpha},\lambda,\bm{a}_0,\bm{b}_0}(\bm{t})} d\bm{t}\right)\frac{d\bm{a}d\bm{b}}{|\bm{a}|_m^2}.
\end{eqnarray*}
Hence,
\begin{eqnarray*}
\F(\bm{a}_0,\bm{b}_0)=\int_{\mathbb{R}^{\N}_0}\int_{\mathbb{R}^{\N}}
\F(\bm{a},\bm{b})K_{\bm{\alpha},\lambda,\mathbf{\Psi}}(\bm{a}_0,\bm{b}_0;\bm{a},\bm{b})\frac{d\bm{a}d\bm{b}}{|\bm{a}|_m^2},
\end{eqnarray*}
where
$\displaystyle K_{\bm{\alpha},\lambda,\mathbf{\Psi}}(\bm{a}_0,\bm{b}_0;\bm{a},\bm{b})=\frac{1}{C_{\bm\alpha,\lambda}}\displaystyle\int_{\mathbb{R}^{\N}}\mathbf{\Psi}_{\bm{\alpha},\lambda,\bm{a},\bm{b}}(\bm{t})\overline{\mathbf{\Psi}_{\bm{\alpha},\lambda,\bm{a}_0,\bm{b}_0}(\bm{t})}  d\bm{t}$.\\
On the other way round, Let's assume that the context holds only for the provided $\F\in L^2\left(\mathbb{R}^{\N}\times \mathbb{R}^{\N}_0, \frac{d\bm{a}d\bm{b}}{|\bm{a}|_m^2}\right)$ then required $f$ is determined by
\begin{align*}
\displaystyle\frac{1}{C_{\bm{\alpha},\lambda}}\int_{\mathbb{R}^{\N}_0}\int_{\mathbb{R}^{\N}}\F(\bm{a},\bm{b})\mathbf{\Psi}_{\bm{\alpha},\lambda,\bm{a},\bm{b}}(\bm{t})\frac{d\bm{a}d\bm{b}}{|\bm{a}|_m^2}.
 \end{align*}
This completes the proof.
\end{proof}

\section{Heisenberg's Uncertainty Inequality for MFrWT}
The uncertainty principle, first proposed by a German physicist named Werner Heisenberg in 1927, states that it is impossible to measure position and momentum at any arbitrarily well localised state at the same time. That is, measuring position without disrupting momentum is impossible, and vice versa. \cite{busch2007heisenberg}. Similarly, in terms of time and frequency, it is impossible to measure time and frequency at any state simultaneously. Here we will derive Heisenberg's uncertainty inequality for MFrFT.
\subsection{Relation between Classical and Fractional Fourier Transform}
Now we derive a relationship between classical FT and MFrFT which will further help us to prove Heisenberg's uncertainity inequality for MFrFT
\begin{align*}
(\mathfrak{F}_{\bm\alpha,\lambda}f)(\bm{\xi})=\int_{\mathbb{R}^{\N}}f(\bm{t})K_{\bm\alpha,\lambda}(\bm{t},\bm{\xi})d\bm{t},
\end{align*}
here $\bm\alpha=(\alpha_1,\alpha_2,\alpha_3,\cdots,\alpha_{\N})$ and $\bm{\xi}=(\xi_1,\xi_2,\xi_3,\cdots,\xi_{\N}).$\\
We have
\begin{eqnarray*}
(\mathfrak{F}_{\bm\alpha,\lambda}f)(\bm{\xi})&=&\int_{\mathbb{R}^{\N}}f(\bm{t})\frac{c(\bm\alpha_{\lambda})}{{(\sqrt{2\pi})^{\N}}}e_{\bm\alpha,\lambda^2}(\bm{\xi})e_{\bm\alpha,\lambda^2}(\bm{t})e^{-i\lambda^2\sum_{k=1}^{\N}t_{k}\xi_{k}\csc\alpha_{k}}d\bm{t}\\
&=&\frac{c(\bm\alpha_{\lambda})}{(\sqrt{2\pi})^{\N}}e_{\bm\alpha,\lambda^2}(\bm{\xi})\int_{\mathbb{R}^{\N}}f(\bm{t})e^{i\lambda^2\sum_{k=1}^{\N}a(\alpha_k){t_{k}}^2}e^{-i\lambda^2\sum_{k=1}^{\N}t_{k}\xi_{k}\csc\alpha_{k}}d\bm{t}.
\end{eqnarray*}
Therefore,
\begin{eqnarray}\label{P1eqn1HUP}
(\mathfrak{F}_{\bm\alpha,\lambda}f)(\bm{\xi})&=&\frac{c(\bm\alpha_{\lambda})}{(\sqrt{2\pi})^{\N}}e_{\bm\alpha,\lambda^2}(\bm{\xi})\int_{\mathbb{R}^{\N}}f(\bm{t})e^{i\lambda^2\sum_{k=1}^{\N}a(\alpha_k){t_{k}}^2}e^{-i\lambda^2\sum_{k=1}^{\N}t_{k}\xi_{k}\csc\alpha_{k}}d\bm{t}.
\end{eqnarray}
As we know traditional Fourier transform is stated as
\begin{eqnarray*}
\mathfrak{F}f(\bm{\xi})&=&\frac{1}{(\sqrt{2\pi})^{\N}}\int_{\mathbb{R}^{\N} }f(\bm{t})e^{-i\sum_{k=1}^{\N}t_{k}\xi_{k}}d\bm{t},\quad \bm{\xi}=(\xi_1,\xi_2,\cdots\xi_{\N})\in\mathbb{R}^{\N}.
\end{eqnarray*}
Replacing $f(\bm{t})$ by $f(\bm{t})e^{i\lambda^2\sum_{k=1}^{\N} a(\alpha_k){t_k}^2}$, and on putting $\frac{ \lambda^2\bm{\xi}}{\sin\bm{\alpha}}$ in place of $\bm{\xi}$, we get
\begin{align}\label{P1eqn2HUP}
\mathfrak{F}(f(\bm{t})e^{i\lambda^2\sum_{k=1}^{{\N}}a(\alpha_k){t_{k}}^2})\left(\frac{ \lambda^2\bm{\xi}}{\sin\bm{\alpha}}\right)
&=\frac{1}{(\sqrt{2\pi})^{\N}}\int_{\mathbb{R}^{\N}}f(\bm{t})e^{i\lambda^2\sum_{k=1}^{\N}a(\alpha_k){t_{k}}^2}e^{-i\lambda^2\sum_{k=1}^{\N}t_{k}\xi_{k}\csc\alpha_{k}}d\bm{t}.
\end{align}
From equation (\ref{P1eqn1HUP}) and (\ref{P1eqn2HUP}), we have
\begin{eqnarray*}
(\mathfrak{F}_{\bm\alpha,\lambda}f)(\bm{\xi})=c(\bm\alpha_\lambda)e_{\bm\alpha,\lambda^2}(\bm{\xi})\mathfrak{F}(f(\bm{t})e^{i\lambda^2\sum_{k=1}^{\N}a(\alpha_k){t_{k}}^2})\left(\frac{ \lambda^2\bm{\xi}}{\sin\bm{\alpha}}\right).
\end{eqnarray*}
This gives
\begin{align}\label{P1eqn2aHUP} 
\mathfrak{F}(f(\bm{t})e^{i\lambda^2\sum_{k=1}^{\N}a(\alpha_k){t_{k}}^2})(\boldsymbol{\eta})
=\frac{1}{c(\bm\alpha_\lambda)e_{\bm\alpha,\lambda^2}\left(\frac{1}{\lambda^2}\boldsymbol{\eta}\sin\bm{\alpha}\right)}(\mathfrak{F}_{\bm\alpha,\lambda}f)\left(\frac{1}{\lambda^2}\boldsymbol{\eta}\sin\bm{\alpha}\right).
\end{align}
Equation (\ref{P1eqn2aHUP}) gives the relation between the classical and fractional FT. For the MFrFT, we now derive Heisenberg's uncertainty principle.\\
If  $f \in L^2(\mathbb{R}^{\N})$ be any arbitrary function, then Heisenberg's uncertainty inequality for dimension  ${\N}$ is described by the following  \cite{verma2021note}.

\begin{eqnarray}{\label{P1eqn3HUP}}
\left(\int_{\mathbb{R}^{\N}} {\|\bm{t}\|}^2\left|{f(\bm{t})}\right|^2d\bm{t}\right)\left(\int_{\mathbb{R}^{\N}}{\|\boldsymbol{\eta}\|}^2|\mathfrak{F}f(\boldsymbol{\eta})|^2d\boldsymbol\eta\right)\geq \frac{{\N}^2}{4} \|f\|^4_{L^2(\mathbb{R}^{\N})}.
\end{eqnarray}
Replace $f(\bm{t})$ by $f(\bm{t})e^{i\lambda^2{\sum}_{k=1}^{\N} a(\alpha_k){t_k}^2}$ and noting that $\left|f(\bm{t})e^{i\lambda^2\sum_{k=1}^{\N}a(\alpha_k){t_k}^2}\right|=|f(\bm{t})|,$ we've
\begin{eqnarray}{\label{P1eqn4HUP}}
\left(\int_{\mathbb{R}^{\N}} {\|\bm{t}\|}^2|{f(\bm{t})|}^2d\bm{t}\right)\left(\int_{\mathbb{R}^{\N}}{\|\boldsymbol{\eta}\|}^2|\mathfrak{F}(f(\bm{t})e^{i\lambda^2\sum_{k=1}^{\N}a(\alpha_k){t_k}^2})(\boldsymbol{\eta})|^2d\boldsymbol\eta\right)\geq \frac{{\N}^2}{4}\|f\|^4_{L^2(\mathbb{R}^{\N})}.
\end{eqnarray}
On putting  $\displaystyle{\boldsymbol{\eta}=\frac{\lambda^2\bm{\xi}}{\sin\bm{\alpha}}}$ in $\displaystyle\int_{\mathbb{R}^{\N}}{\|\boldsymbol{\eta}\|}^2|\mathfrak{F}(f(\bm{t})e^{i\lambda^2\sum_{k=1}^{\N}a(\alpha_k){t_k}^2})(\boldsymbol{\eta})|^2d\boldsymbol\eta$ and using equation (\ref{P1eqn2aHUP}), we get
\begin{eqnarray*}
\int_{\mathbb{R}^{\N}}{\|\boldsymbol{\eta}\|}^2|\mathfrak{F}(f(\bm{t})e^{i\lambda^2\sum_{k=1}^{\N}a(\alpha_k){t_k}^2})(\boldsymbol{\eta})|^2d\boldsymbol\eta =\frac{\lambda^{2N}}{|\sin\bm{\alpha}|_m}\int_{\mathbb{R}^{\N}}{\left\|\frac{\lambda^2\bm{\xi}}{\sin\bm\alpha}\right\|^2}\left|\frac{1}{c(\bm\alpha_\lambda)e_{\alpha,\lambda^2}(\bm{\xi})}(\mathfrak{F}_{\bm\alpha,\lambda} f)(\bm{\xi})\right|^2 d\bm{\xi}.
\end{eqnarray*}
Now, consider $\displaystyle{\left\|\frac{\lambda^2\bm{\xi}}{\sin\bm\alpha}\right\|}=\left\|\left(\frac{\lambda^2\xi_1}{\sin\alpha_1},\frac{\lambda^2\xi_2}{\sin\alpha_2},\cdots,\frac{\lambda^2\xi_{\N}}{\sin\alpha_{\N}}\right)\right\|$.\\
Let $\displaystyle\frac{1}{\sin\alpha_i}=m_i$ for $i=1,2,\cdots,{\N},$ then
\begin{align}\label{p1eqn5HUP} \notag
{\left\|\frac{\lambda^2\bm{\xi}}{\sin\bm\alpha}\right\|}&=\lambda^2\|\left(m_1\xi_1,m_2\xi_2,\cdots,m_{\N}\xi_{\N}\right)\| \\ \notag
&=\lambda^2\left({m_1}^2\xi_1^2+{m_2}^2\xi_2^2+\cdots+{m_{\N}}^2\xi_{\N}^2\right)^{\frac{1}{2}}.\\ 
\end{align}
Therefore,
\begin{align}{\label{p1eqn6HUP}}
\left\|\frac{\lambda^2\bm{\xi}}{\sin\bm\alpha}\right\|\leq \lambda^2M\|\bm{\xi}\|,
\end{align}
where $M^2=\max\{m_i^2:i=1,2\cdots,{\N}\}.$\\
Therefore,
\begin{eqnarray}{\label{p1eqn7HUP}}
\int_{\mathbb{R}^{\N}}{\|\boldsymbol{\eta}\|}^2|\mathfrak{F}(f(\bm{t})e^{i\lambda^2\sum_{k=1}^{\N}a(\alpha_k){t_k}^2})(\boldsymbol{\eta})|^2d\boldsymbol\eta \leq \frac{M^2}{|c(\bm\alpha_\lambda)|^2}\frac{\lambda^{2N+4}}{\left|\sin\bm\alpha\right|_m}\int_{\mathbb{R}^{\N}}{\|\bm{\xi}\|}^2|\mathfrak{F}_{\bm\alpha,\lambda}f(\bm{\xi})|^2d\bm{\xi}.
\end{eqnarray}
Using inequality (\ref{p1eqn7HUP}) in (\ref{P1eqn4HUP}), we get
\begin{eqnarray*}
\left(\int_{\mathbb{R}^{\N}} {\|\bm{t}\|}^2|{f(\bm{t})|}^2d\bm{t}\right)\left(\int_{\mathbb{R}^{\N}}{\|\bm{\xi}\|}^2|\mathfrak{F}_{\bm\alpha,\lambda}f(\bm{\xi})|^2d\bm{\xi}\right)\geq \frac{|c(\bm\alpha_\lambda)|^2|\sin\bm\alpha|_m}{M^2(\lambda^2)^{{\N}+2}}\frac{{\N}^2}{4}\|f\|^4_{L^2(\mathbb{R}^{\N})}.
\end{eqnarray*}
Which implies
\begin{eqnarray}{\label{P1eqn8HUP}}
\left(\int_{\mathbb{R}^{\N}} {\|\bm{t}\|}^2|{f(\bm{t})|}^2d\bm{t}\right)\left(\int_{\mathbb{R}^{\N}}{\|\bm{\xi}\|}^2|\mathfrak{F}_{\bm\alpha,\lambda}f(\bm{\xi})|^2d\bm{\xi}\right)\geq P_{\bm\alpha,\lambda}\frac{{\N}^2}{4}\|f\|^4_{L^2(\mathbb{R}^{\N})},
\end{eqnarray}
where
\begin{align} \label{P1eqn5aHUP}
P_{\bm\alpha,\lambda}=\frac{|c(\bm\alpha_\lambda)|^2|\sin\bm\alpha|_m}{M^2(\lambda^2)^{{\N}+2}}.
\end{align}

\begin{lemma}\label{lemofHUP}
Let $\psi$ be an admissible wavelet and for any function $f\in L^2(\mathbb{R}^{\N}).$ Then,
\begin{eqnarray*}
\int_{\mathbb{R}_0^{\N}}\int_{\mathbb{R}^{\N}} \|\bm{\xi}\|^2|\mathfrak{F}_{\bm\alpha,\lambda} ( (W_{\psi}^{\bm\alpha,\lambda}{f})(\bm{a},\cdot))(\bm{\xi})|^2d\bm{\xi}\frac{d\bm{a}}{|\bm{a}|_m^2}&=&C_{\bm\alpha,\lambda}\int_{\mathbb{R}^{\N}}\|\bm{\xi}\|^2|(\mathfrak{F}_{\bm\alpha,\lambda}f)(\bm{\xi})|^2d\bm{\xi},
\end{eqnarray*}
where $C_{\bm\alpha,\lambda}$ is as defined in equation (\ref{P1eqn11}).
\end{lemma}
\begin{proof}
Using lemma (\ref{P1lemma1}) and theorem (\ref{P1innTh}), we get
\begin{align*}
\int_{\mathbb{R}_0^{\N}}\int_{\mathbb{R}^{\N}} \|\bm{\xi}\|^2&|\mathfrak{F}_{\bm\alpha,\lambda} ((W_{\psi}^{\bm\alpha,\lambda}{f})(\bm{a},\bm{b}))(\bm{\xi})|^2d\bm{\xi}\frac{d\bm{a}}{|\bm{a}|_m^2}\\
&=\int_{\mathbb{R}_0^{\N}}\int_{\mathbb{R}^{\N}} \|\bm{\xi}\|^2{|\bm{a}|_m({2\pi})^{\N}}(\mathfrak{F}_{\bm\alpha,\lambda}\psi)(\bm{a}\bm{\xi})(\mathfrak{F}_{\bm\alpha,\lambda}f)(\bm{\xi})\overline{(\mathfrak{F}_{\bm\alpha,\lambda} \psi)(\bm{a}\bm{\xi})}\overline{\mathfrak{F}_{\bm\alpha,\lambda}f(\bm{\xi})}d\bm{\xi}\frac{d\bm{a}}{|\bm{a}|_m^2}\\
&=C_{\bm\alpha,\lambda}\int_{\mathbb{R}^{\N}}\|\bm{\xi}\|^2|(\mathfrak{F}_{\bm\alpha,\lambda}f)(\bm{\xi})|^2d\bm{\xi}.
\end{align*}
This concludes the proof.
\end{proof}
Equation (\ref{P1eqn8HUP}) is the Heisenberg's uncertainty inequality for the MFrFT. Now we prove the Heisenberg's uncertainty principle for the MFrWT following the idea of \cite{singer1999uncertainty},\cite{wilczok2000new}.
\begin{theorem}
(Heisenberg’s inequality of uncertainty for MFrWT) Let $\psi$ be an admissible wavelet and for any function $f\in L^2(\mathbb{R}^{\N}),$ the MFrWT of $f$ with respect to the variable $\bm\alpha$ is specified by $(W_{\psi}^{\bm\alpha,\lambda}{f})(\bm{a},\bm{b})$ . Then
\begin{align*}
\left(\int_{\mathbb{R}_0^{\N}}\int_{\mathbb{R}^{\N}}\|\bm{b}\|^2|(W_{\psi}^{\bm\alpha,\lambda}{f})(\bm{a},\bm{b})|^2d\bm{b}\frac{d\bm{a}}{|\bm{a}|_m^2}\right)\left(\int_{\mathbb{R}^{\N}} \|\bm{\xi}\|^2|(\mathfrak{F}_{\bm\alpha,\lambda}f)(\bm{\xi})|^2d\bm{\xi}\right) \geq P_{\bm\alpha,\lambda}C_{\bm\alpha,\lambda}\frac{{\N}^2}{4}\|f\|^2_{L^2(\mathbb{R}^{\N})},
\end{align*}
where $C_{\bm\alpha,\lambda}$ and $P_{\bm\alpha,\lambda}$ are defined in equation (\ref{P1eqn11}) and equation (\ref{P1eqn5aHUP}) respectively.
\end{theorem}
\begin{proof}
By substituting $(W_{\psi}^{\bm\alpha,\lambda}{f})(\bm{a},\cdot)$ for $f(\cdot)$ in equation (\ref{P1eqn8HUP}), we obtain
\begin{align*}
\hspace{-1cm}\left(\int_{\mathbb{R}^{\N}}\|\bm{b}\|^2|(W_{\psi}^{\bm\alpha,\lambda}{f})(\bm{a},\bm{b})|^2d\bm{b}\right)\left(\int_{\mathbb{R}^{\N}} \|\bm{\xi}\|^2|\mathfrak{F}_{\bm\alpha,\lambda} ((W_{\psi}^{\bm\alpha,\lambda}{f})(\bm{a},.))(\bm{\xi})|^2d\bm{\xi} \right) \geq P_{\bm\alpha,\lambda}\frac{{\N}^2}{4}\|(W_{\psi}^{\bm\alpha,\lambda}{f})(\bm{a},.)\|^4_{L^2(\mathbb{R}^{\N})}.
\end{align*}
Taking the square root of the above equation on both sides yields,
\begin{align*}
\hspace{-0.8cm}\left(\int_{\mathbb{R}^{\N}}\|\bm{b}\|^2|(W_{\psi}^{\bm\alpha,\lambda}{f})(\bm{a},\bm{b})|^2d\bm{b}\right)^{\frac{1}{2}}\left(\int_{\mathbb{R}^{\N}} \|\bm{\xi}\|^2|\mathfrak{F}_{\bm\alpha,\lambda} ((W_{\psi}^{\bm\alpha,\lambda}{f})(\bm{a},\bm{b}))(\bm{\xi})|^2d\bm{\xi}\right)^{\frac{1}{2}} \geq {P^{\frac{1}{2}}_{\bm\alpha,\lambda}}\frac{{\N}}{2}\int_{\mathbb{R}^{\N}}|(W_{\psi}^{\bm\alpha,\lambda}{f})(\bm{a},\bm{b})|^2d\bm{b}.
\end{align*}
Now, we Integrate above equation with respect to the measure $\frac{d\bm{a}}{|\bm{a}|_m^2}$, we've
\begin{align*}
\hspace{-2cm}\int_{\mathbb{R}_0^{\N}}\left(\int_{\mathbb{R}^{\N}}\|\bm{b}\|^2|(W_{\psi}^{\bm\alpha,\lambda}{f})(\bm{a},\bm{b})|^2d\bm{b}\right)^{\frac{1}{2}}\left(\int_{\mathbb{R}^{\N}} \|\bm{\xi}\|^2|\mathfrak{F}_{\bm\alpha,\lambda} (W_{\psi}^{\bm\alpha,\lambda}{f}(\bm{a},\bm{b}))(\bm{\xi})|^2d\bm{\xi}\right)^{\frac{1}{2}}\frac{d\bm{a}}{|\bm{a}|_m^2}\\ \geq {P^{\frac{1}{2}}_{\bm\alpha,\lambda}}\frac{{\N}}{2}\int_{\mathbb{R}_0^{\N}} \int_{\mathbb{R}^{\N}}|(W_{\psi}^{\bm\alpha,\lambda}{f})(\bm{a},\bm{b})|^2\frac{d\bm{a}d\bm{b}}{|\bm{a}|_m^2}.
\end{align*}
Using Holder's inequality and corollary (\ref{corofinnthe}), we have
\begin{align*}
\left(\int_{\mathbb{R}_0^{\N}}\int_{\mathbb{R}^{\N}}\|\bm{b}\|^2|(W_{\psi}^{\bm\alpha,\lambda}{f})(\bm{a},\bm{b})|^2d\bm{b}\frac{d\bm{a}}{|\bm{a}|_m^2}\right)^{\frac{1}{2}}\left(\displaystyle\int_{\mathbb{R}_0^{\N}}\int_{\mathbb{R}^{\N}}\|\bm{\xi}\|^2|\mathfrak{F}_{\bm\alpha,\lambda} ((W_{\psi}^{\bm\alpha,\lambda}{f})(\bm{a},\bm{b}))(\bm{\xi})|^2d\bm{\xi} \frac{d\bm{a}}{|\bm{a}|_m^2}\right)^{\frac{1}{2}} \\
\geq {P^{\frac{1}{2}}_{\bm\alpha,\lambda}}\frac{{\N}}{2}C_{\bm\alpha,\lambda}\|f\|^2_{L^2(\mathbb{R}^{\N})}.
\end{align*}
Squaring both sides and applying lemma (\ref{lemofHUP}), we obtain
\begin{align*}
\left(\int_{\mathbb{R}_0^{\N}}\int_{\mathbb{R}^{\N}}\|\bm{b}\|^2|(W_{\psi}^{\bm\alpha,\lambda}{f})(\bm{a},\bm{b})|^2d\bm{b}\frac{d\bm{a}}{|\bm{a}|_m^2}\right)\left(C_{\bm\alpha,\lambda}\int_{\mathbb{R}^{\N}}\|\bm{\xi}\|^2|(\mathfrak{F}_{\bm\alpha,\lambda}f)(\bm{\xi})|^2d\bm{\xi}\right)\geq P_{\bm\alpha,\lambda} \frac{{\N}^2}{4}C^2_{\bm\alpha,\lambda}\|f\|^4_{L^2(\mathbb{R}^{\N})},
\end{align*}
i.e.,
\begin{align*}
\left(\int_{\mathbb{R}_0^{\N}}\int_{\mathbb{R}^{\N}}\|\bm{b}\|^2|(W_{\psi}^{\bm\alpha,\lambda}{f})(\bm{a},\bm{b})|^2d\bm{b}\frac{d\bm{a}}{|\bm{a}|_m^2}\right)\left(\int_{\mathbb{R}^{\N}}\|\bm{\xi}\|^2|(\mathfrak{F}_{\bm\alpha,\lambda}f)(\bm{\xi})|^2d\bm{\xi}\right)\geq P_{\bm\alpha,\lambda} C_{\bm\alpha,\lambda}\frac{{\N}^2}{4}\|f\|^4_{L^2(\mathbb{R}^{\N})}.
\end{align*}
This concludes the proof for this section.
\end{proof}
\subsection{Logarithmic Uncertainty Principle for MFrWT}
Basic idea about logarithmic uncertainty principle for fractional wavelet transform is given by Mawardi Bahri in \cite{bahri2017logarithmic}.
In this context, we formally derive a logarithmic uncertainty principle for the MFrFT using MFrFT properties and the logarithmic uncertainty principle for the fractional Fourier transform.\\
In \cite{beckner1995pitt}, the logarithmic uncertainty inequality for classical Fourier transform is given by
\begin{align}{\label{P1eqn1LgUP}}
\int_{\mathbb{R}^{\N}} \ln\|\mathbf{x}\||f(\mathbf{x})|^2d\mathbf{x}+\int_{\mathbb{R}^{\N}} \ln\|\boldsymbol\eta\||\mathfrak{F}f(\boldsymbol{\eta})|^2d\boldsymbol\eta \geq D\int_{\mathbb{R}^{\N}}|f(\mathbf{x})|^2d\mathbf{x},
\end{align} 
where $D = \psi\left(\frac{{\N}}{4}\right)-\ln 2,$ \quad $\psi(t)=\frac{d}{dt}[\ln\Gamma(t)]$.\\

Replacing $f(\mathbf{x})$ with $f(\mathbf{x})e^{i\lambda^2\sum_{k=1}^{\N}a(\alpha_k){x_k}^2}$ in inequality (\ref{P1eqn1LgUP}) and observing that  $\left|f(\mathbf{x})e^{i\lambda^2\sum_{k=1}^{\N}a(\alpha_k){x_k}^2}\right|=\left|f(\mathbf{x})\right|,$ we get
\begin{eqnarray}\label{P1eqn1aLgUP}
\int_{\mathbb{R}^{\N}} \ln\|\mathbf{x}\|\left|f(\mathbf{x})\right|^2d\mathbf{x}+\int_{\mathbb{R}^{\N}} \ln\|\boldsymbol\eta\||\mathfrak{F}(f(\mathbf{x})e^{i\lambda^2\sum_{k=1}^{\N}a(\alpha_k){x_k}^2})(\boldsymbol{\eta})|^2d\boldsymbol\eta \geq D\int_{\mathbb{R}^{\N}}|f(\mathbf{x})|^2d\mathbf{x}.
\end{eqnarray}
Put $\boldsymbol{\eta}=\frac{\lambda^2\bm{\xi}}{\sin\bm{\alpha}}$ in $\displaystyle{\int_{\mathbb{R}^{\N}} \ln\|\boldsymbol\eta\||\mathfrak{F}f(\mathbf{x})e^{i\lambda^2\sum_{k=1}^{\N}a(\alpha_k){x_k}^2}(\boldsymbol{\eta})|^2d\boldsymbol\eta}$ and using equation (\ref{P1eqn2aHUP}), we get
\begin{align*}
\int_{\mathbb{R}^{\N}} \ln\|\boldsymbol\eta\||\mathfrak{F}(f(\mathbf{x})e^{i\lambda^2\sum_{k=1}^{\N}a(\alpha_k){x_k}^2})(\boldsymbol{\eta})|^2d\boldsymbol\eta 
&=\frac{\lambda^{2N}}{\left|\sin\bm\alpha\right|_m}\int_{\mathbb{R}^{\N}} \ln\left\|\frac{\lambda^2\bm{\xi}}{\sin\bm{\alpha}}\right\|\left|\frac{1}{c(\bm\alpha_\lambda)e_{\bm\alpha,\lambda^2}(\bm{\xi})}\mathfrak{F}_{\bm\alpha,\lambda}f(\bm{t})(\bm{\xi})\right|^2 d\bm{\xi}\\
&=\frac{1}{|c(\bm\alpha_\lambda)|^2}\frac{\lambda^{2N}}{\left|\sin\bm\alpha\right|_m}\int_{\mathbb{R}^{\N}}\ln\left\|\frac{\lambda^2\bm{\xi}}{\sin\bm{\alpha}}\right\|\left|(\mathfrak{F}_{\bm\alpha,\lambda}f)(\bm{\xi})\right|^2 d\bm{\xi}.
\end{align*}
From equation (\ref{P1eqn1aLgUP}), we get
\begin{align*}
\int_{\mathbb{R}^{\N}} \ln\|\mathbf{x}\||f(\mathbf{x})|^2d\mathbf{x}+\frac{1}{|c(\bm\alpha_\lambda)|^2}\frac{\lambda^{2N}}{\left|\sin\bm\alpha\right|_m}\int_{\mathbb{R}^{\N}}\ln
\left\|\frac{\lambda^2\bm{\xi}}{\sin\bm{\alpha}}\right\| |(\mathfrak{F}_{\bm\alpha,\lambda}f)(\bm{\xi})|^2d\bm{\xi}
\geq D\int_{\mathbb{R}^{\N}}|f(\mathbf{x})|^2d\mathbf{x}.
\end{align*}
Then we have
\begin{align}\label{P1eqn2LgUP}
\int_{\mathbb{R}^{\N}} \ln\|\mathbf{x}\||f(\mathbf{x})|^2d\mathbf{x}+P'_{\bm\alpha,\lambda}\int_{\mathbb{R}^{\N}}\ln\left\|\frac{\lambda^2\bm{\xi}}{\sin\bm{\alpha}}\right\||(\mathfrak{F}_{\bm\alpha,\lambda}f)(\bm{\xi})|^2d\bm{\xi}\geq D\int_{\mathbb{R}^{\N}}|f(\mathbf{x})|^2d\mathbf{x},
\end{align}
where 
\begin{equation}\label{P1eqnPalpha}
P'_{\bm\alpha,\lambda}=\frac{\lambda^{2N}}{{|c(\bm\alpha_\lambda)|^2}|\sin\bm\alpha|_m}.
\end{equation} 
From equation (\ref{p1eqn6HUP}), we know
\begin{align*}
\left\|\frac{\lambda^2\bm{\xi}}{\sin\bm\alpha}\right\|\leq \lambda^2M\|\bm{\xi}\|.
\end{align*} 
Taking $\ln$ on both sides of above equation, we get
\begin{align*}
 \ln\left\|\frac{\lambda^2\bm{\xi}}{\sin\bm{\alpha}}\right\|\leq \ln (\lambda^2M) + \ln\|\bm{\xi}\|.
\end{align*}
Therefore, equation(\ref{P1eqn2LgUP}) can be re-written as,
\begin{align*}
\int_{\mathbb{R}^{\N}} \ln\|\mathbf{x}\||f(\mathbf{x})|^2d\mathbf{x}+P'_{\bm\alpha,\lambda}\int_{\mathbb{R}^{\N}}[\ln (\lambda^2M) + \ln\|\bm{\xi}\|]|(\mathfrak{F}_{\bm\alpha,\lambda}f)(\bm{\xi})|^2d\bm{\xi}\geq D\int_{\mathbb{R}^{\N}}|f(\mathbf{x})|^2d\mathbf{x}\\
\int_{\mathbb{R}^{\N}} \ln\|\mathbf{x}\||f(\mathbf{x})|^2d\mathbf{x}+P'_{\bm\alpha,\lambda}\ln (\lambda^2M)\int_{\mathbb{R}^{\N}}|(\mathfrak{F}_{\bm\alpha,\lambda}f)(\bm{\xi})|^2d\bm{\xi} +P'_{\bm\alpha,\lambda} \int_{\mathbb{R}^{\N}}\ln\|\bm{\xi}\||(\mathfrak{F}_{\bm\alpha,\lambda}f)(\bm{\xi})|^2d\bm{\xi}\geq D\int_{\mathbb{R}^{\N}}|f(\mathbf{x})|^2d\mathbf{x}.
\end{align*}
Using Parseval's relation, we get
\begin{align}\label{P1eqn3LgUP}
\int_{\mathbb{R}^{\N}} \ln\|\mathbf{x}\||f(\mathbf{x})|^2d\mathbf{x}+P'_{\bm\alpha,\lambda}\int_{\mathbb{R}^{\N}}\ln\|\bm{\xi}\||(\mathfrak{F}_{\bm\alpha,\lambda}f)(\bm{\xi})|^2d\bm{\xi}\geq (D-P'_{\alpha,\lambda}\ln (\lambda^2M))\int_{\mathbb{R}^{\N}}|f(\mathbf{x})|^2d\mathbf{x}.
\end{align}

With the help of logarithmic uncertainty inequality for MFrFT (\ref{P1eqn3LgUP}), we will prove the following for the MFrWT.
\begin{theorem}(Logarithmic uncertainty inequality for MFrWT)
Let $\psi$ be an admissible wavelet and for any $f\in L^2(\mathbb{R}^{\N})$ and for some $M> 0,$  Let MFrWT of $f$ with  parameter $\bm{\alpha} $ is given by $(W_{\psi}^{\bm{\alpha}, \lambda}f)(\bm{a},\bm{b}).$ Then
\begin{align*}
\int_{\mathbb{R}_0^{\N}}\int_{\mathbb{R}^{\N}}\ln\|\bm{b}\||(W_{\psi}^{\bm\alpha,\lambda}{f})(\bm{a},\bm{b})|^2d\bm{b}\frac{d\bm{a}}{|\bm{a}|_m^2}+C_{\bm\alpha,\lambda}P'_{\bm\alpha,\lambda}\int_{\mathbb{R}^{\N}}\ln\|\bm{\xi}\||&\mathfrak{F}_{\bm\alpha,\lambda}f)|^2(\bm{\xi})d\bm{\xi}\\
&\geq\left(D-P'_{\boldsymbol{\alpha,\lambda}}\ln (\lambda^2M)\right)C_{\alpha,\lambda}\|f\|_{L^2(\mathbb{R}^{\N})}^2,
\end{align*}
where $P'_{\bm\alpha,\lambda}$ is given by equation (\ref{P1eqnPalpha}).
\end{theorem}
\begin{proof}
Replace $f(.)$ with $(W_{\psi}^{\bm\alpha,\lambda}{f})(\bm{a},.)$ on either side of equation (\ref{P1eqn3LgUP}), we have
\begin{align*}
\int_{\mathbb{R}^{\N}}\ln\|\bm{b}\||(W_{\psi}^{\bm\alpha,\lambda}{f})(\bm{a},\bm{b})|^2d\bm{b}+P'_{\bm\alpha,\lambda}\int_{\mathbb{R}^{\N}} \ln\left\|\bm{\xi}\right\||(\mathfrak{F}_{\bm\alpha,\lambda}(W_{\psi}^{\bm\alpha,\lambda}{f})&(\bm{a},\bm{b}))(\bm{\xi})|^2d\bm{\xi}\\
&\geq\left(D-P'_{\boldsymbol{\alpha,\lambda}}\ln (\lambda^2M)\right)\int_{\mathbb{R}^{\N}}|(W_{\psi}^{\bm\alpha,\lambda}{f})(\bm{a},\bm{b})|^2d\bm{b}.
\end{align*}
Now, we integrate above equation with respect to measure $\frac{d\bm{a}}{|\bm{a}|_m^2},$ we get
\begin{align}\label{P1eqn4LgUP}\notag
\int_{\mathbb{R}_0^{\N}}\int_{\mathbb{R}^{\N}}\ln\|\bm{b}\||(W_{\psi}^{\bm\alpha,\lambda}{f})(\bm{a},\bm{b})|^2d\bm{b}\frac{d\bm{a}}{|\bm{a}|_m^2}+P'_{\bm\alpha,\lambda}\int_{\mathbb{R}_0^{\N}}\int_{\mathbb{R}^{\N}}&\ln\left\|\bm{\xi}\right\||(\mathfrak{F}_{\bm\alpha,\lambda}(W_{\psi}^{\bm\alpha,\lambda}{f})(\bm{a},\bm{b}))(\bm{\xi})|^2d\bm{\xi}\frac{d\bm{a}}{|\bm{a}|_m^2}\\
&\geq\left(D-P'_{\boldsymbol{\alpha,\lambda}}\ln (\lambda^2M)\right)\int_{\mathbb{R}_0^{\N}}\int_{\mathbb{R}^{\N}}|(W_{\psi}^{\bm\alpha,\lambda}{f})(\bm{a},\bm{b})|^2d\bm{b}\frac{d\bm{a}}{|\bm{a}|_m^2}.
\end{align}

Hence, using corollary (\ref{corofinnthe}) and lemma (\ref{P1lemma1}) in equation (\ref{P1eqn4LgUP}), we have
\begin{align*}
 \int_{\mathbb{R}_0^{\N}}\int_{\mathbb{R}^{\N}}\ln\|\bm{b}\||(W_{\psi}^{\bm\alpha,\lambda}{f})(\bm{a},\bm{b})|^2d\bm{b}\frac{d\bm{a}}{|\bm{a}|_m^2}+P'_{\bm\alpha,\lambda}C_{\bm\alpha,\lambda}\int_{\mathbb{R}^{\N}}\ln\|\bm{\xi}\|&|(\mathfrak{F}_{\bm\alpha,\lambda}f)(\bm{\xi})|^2d\bm{\xi}\\
 &\geq\left(D-P'_{\boldsymbol{\alpha,\lambda}}\ln (\lambda^2M)\right)C_{\alpha,\lambda}\int_{\mathbb{R}^{\N}}|f(\mathbf{x})|^2d\mathbf{x}.
\end{align*}
where $M^2=\displaystyle\max \left\{\frac{1}{\sin^2\alpha_i}:i=1,2,\cdots,{\N} \right \}.$\\
This concludes the proof for this section.
\end{proof}
\subsection{Local Uncertainty Inequality for MFrWT}
In harmonic analysis, Heisenberg's uncertainty principle states that if any arbitrary function $f$ is specifically restricted, the corresponding Fourier transform of $f$ cannot be reduced to the neighborhood of a point, But this never prevents the Fourier transform of $f$ from being focused approximately at two or more distinct points.  In fact, it cannot occur either, and this is the purpose of local uncertainty inequalities to make this specific. The local uncertainty principle for the classical Fourier transform can be stated using the definition of the traditional Fourier transform for $\alpha=\frac{\pi}{2},$ as follows \cite{yang2013mathematical} :\\
(i)There exists a invariable $A_\p$ for  $0 < \p < \frac{{\N}}{2}$ such that for every measurable subsets $K$ of $\mathbb{R}^{\N}$ and for all $f \in L^2(\mathbb{R}^{\N})$
\begin{eqnarray}{\label{P1(i)LUP}}
\int_{K}|(\mathfrak{F}f)(\bm{\xi})|^2d\bm{\xi}\leq A_\p(\lambda(K))^{\frac{2\p}{\N}}\|\|\mathbf{x}\|^\p f\|^2_{L^2(\mathbb{R}^{\N})}.
\end{eqnarray}
(ii)There exists a invariable $A_\p$ for  $\p>\frac{\N}{2}$ such that for every measurable subsets $K$ of $\mathbb{R}^{\N}$ and for all $f \in L^2(\mathbb{R}^{\N})$
\begin{eqnarray}{\label{P1(ii)LUP}}
\int_{K}|(\mathfrak{F}f)(\bm{\xi})|^2d\bm{\xi}\leq A_\p\lambda(K)\|f\|^{2-\frac{\N}{\p}}_{L^2(\mathbb{R}^{\N})}\|\|\mathbf{x}\|^\p f\|^{\frac{\N}{\p}}_{L^2(\mathbb{R}^{\N})}.
\end{eqnarray}
Replacing $ f(\mathbf{x})$ by $f(\mathbf{x})e^{i\lambda^2\sum_{k=1}^{\N}a(\alpha_k){x_{k}}^2}$ in equation (\ref{P1(i)LUP}) and using equation (\ref{P1eqn2aHUP}) we have
\begin{eqnarray}\label{P1eqn1LUP}
\int_{K}\left|(\mathfrak{F}_{\bm\alpha,\lambda}f)\left(\frac{1}{\lambda^2}\bm{\xi}\sin\bm{\alpha}\right)\right|^2d\bm{\xi}\leq |c(\bm\alpha_\lambda)|^2A_\p(\lambda(K))^{\frac{2\p}{\N}}\|\|\mathbf{x}\|^\p f\|^2_{L^2(\mathbb{R}^{\N})}.
\end{eqnarray}
Put $\frac{1}{\lambda^2}\bm{\xi}\sin\bm{\alpha}=\boldsymbol{\eta}$ in equation (\ref{P1eqn1LUP}), we have
\begin{eqnarray*}
\frac{(\lambda^2)^{\N}}{|\sin\bm\alpha|_m}\int_K|(\mathfrak{F}_{\bm\alpha,\lambda}f)(\boldsymbol\eta)|^2d\boldsymbol\eta\leq |c(\bm\alpha_\lambda)|^2 A_\p(\lambda(K))^{\frac{2\p}{\N}}\|\|\mathbf{x}\|^\p f\|^2_{L^2(\mathbb{R}^{\N})}.
\end{eqnarray*}
That is,
\begin{eqnarray}\label{P1eqn2LUP} 
\int_K|(\mathfrak{F}_{\bm\alpha,\lambda}f)(\boldsymbol\eta)|^2d\boldsymbol\eta \leq |c(\bm\alpha_\lambda)|^2 A_\p\frac{|\sin\bm\alpha|_m}{(\lambda^2)^{\N}}(\lambda(K))^{\frac{2\p}{\N}}\|\|\mathbf{x}\|^\p f\|^2_{L^2(\mathbb{R}^{\N})}.
\end{eqnarray}
Similarly, from equation (\ref{P1(ii)LUP}), we get
\begin{eqnarray}\label{P1eqn3LUP}
\int_K|(\mathfrak{F}_{\bm\alpha,\lambda}f)(\boldsymbol\eta)|^2d\boldsymbol\eta \leq |c(\bm\alpha_\lambda)|^2 A_\p\frac{|\sin\bm\alpha|_m}{(\lambda^2)^{\N}}\lambda(K)\|f\|^{2-\frac{\N}{\p}}_{L^2(\mathbb{R}^{\N})}\|\|\mathbf{x}\|^\p f\|^{\frac{\N}{\p}}_{L^2(\mathbb{R}^{\N})}.
\end{eqnarray}
Now, we will illustrate the major outcome of this section.
\begin{theorem}\label{P1Th1LUP}
(Local uncertainty inequality for MFrWT) If $\psi$  is an admissible wavelet and $\p$ be such that $0 < \p < \frac{\N}{2}.$ Then, for every measurable subsets $K$ of $\mathbb{R}^{\N},$ there exist constant $A_\p$  and for all $f\in L^2(\mathbb{R}^{\N})$ such that 

\begin{eqnarray*}
\int_K|(\mathfrak{F}_{\bm\alpha,\lambda}f)(\boldsymbol\eta)|^2d\boldsymbol\eta\leq\frac{|c(\bm\alpha_\lambda)|^2}{C_{\bm\alpha,\lambda}}A_\p\frac{|\sin\bm\alpha|_m}{(\lambda^2)^{\N}}(\lambda(K))^{\frac{2\p}{\N}}\int_{\mathbb{R}_0^{\N}}\int_{\mathbb{R}^{\N}}\|\bm{b}\|^{2\p} \left|(W_{\psi}^{\bm\alpha,\lambda}{f})(\bm{a},\bm{b})\right|^2d\bm{b}\frac{d\bm{a}}{|\bm{a}|_m^2}.
\end{eqnarray*}
\end{theorem}
\begin{proof}
By substituting $W_{\psi}^{\bm\alpha,\lambda}{f}(\bm{a},\cdot)$ for $f(\cdot)$ in equation (\ref{P1eqn2LUP}), we obtain

\begin{eqnarray*}
\int_{K}|\mathfrak{F}_{\bm\alpha,\lambda}((W_{\psi}^{\bm\alpha,\lambda}{f})(\bm{a},\bm{b}))(\boldsymbol\eta)|^2d\boldsymbol\eta \leq |c(\bm\alpha_\lambda)|^2 A_\p\frac{|\sin\bm\alpha|_m}{(\lambda^2)^{\N}}(\lambda(K))^{\frac{2\p}{\N}}\|\|\bm{b}\|^\p (W_{\psi}^{\bm\alpha,\lambda}{f})(\bm{a},\bm{b})\|^2_{L^2(\mathbb{R}^{\N})}.
\end{eqnarray*}
Now we integrate either side of above equation with respect to  $\frac{d\bm{a}}{|\bm{a}|_m^2},$ we get
\begin{eqnarray*}
\int_{\mathbb{R}_0^{\N}}\int_{K}|\mathfrak{F}_{\bm\alpha,\lambda}((W_{\psi}^{\bm\alpha,\lambda}{f})(\bm{a},\bm{b}))(\boldsymbol\eta)|^2\frac{d\bm{a}}{|\bm{a}|_m^2}d\boldsymbol\eta\leq|c(\bm\alpha_\lambda)|^2 A_\p\frac{|\sin\bm\alpha|_m}{(\lambda^2)^{\N}}(\lambda(K))^{\frac{2\p}{\N}}\int_{\mathbb{R}_0^{\N}}\|\|\bm{b}\|^\p (W_{\psi}^{\bm\alpha,\lambda}{f})(\bm{a},\bm{b})\|^2_{L^2(\mathbb{R}^{\N})}\frac{d\bm{a}}{|\bm{a}|_m^2},
\end{eqnarray*}
and using theorem (\ref{P1innTh}),
\begin{eqnarray*}
{C_{\bm\alpha,\lambda}}\int_{K}|(\mathfrak{F}_{\bm\alpha,\lambda}f)|^2d\boldsymbol\eta\leq |c(\bm\alpha_\lambda)|^2 A_\p \frac{|\sin\bm\alpha|_m}{(\lambda^2)^{\N}}(\lambda(K))^{\frac{2\p}{\N}}\int_{\mathbb{R}_0^{\N}}\int_{\mathbb{R}^{\N}}\left|\|\bm{b}\|^\p (W_{\psi}^{\bm\alpha,\lambda}{f})(\bm{a},\bm{b})\right|^2d\bm{b}\frac{d\bm{a}}{|\bm{a}|_m^2}.
\end{eqnarray*}
This gives
\begin{eqnarray*}
\int_{K}|(\mathfrak{F}_{\bm\alpha,\lambda}f)|^2d\boldsymbol\eta\leq \frac{|c(\bm\alpha_\lambda)|^2}{C_{\bm\alpha,\lambda}} A_\p \frac{|\sin\bm\alpha|_m}{(\lambda^2)^{\N}}(\lambda(K))^{\frac{2\p}{\N}}\int_{\mathbb{R}_0^{\N}}\int_{\mathbb{R}^{\N}}\|\bm{b}\|^{2\p} \left|(W_{\psi}^{\bm\alpha,\lambda}{f})(\bm{a},\bm{b})\right|^2d\bm{b}\frac{d\bm{a}}{|\bm{a}|_m^2}.
\end{eqnarray*}
This concludes the proof.
\end{proof}
The following theorem expresses the local uncertainty inequality for MFrWT when $\p \textgreater \frac{\N}{2}.$
\begin{theorem}
(Local uncertainty inequality for MFrWT ) Assume $\psi$ is an admissible wavelet and $\p$ is such that $\p \textgreater \frac{\N}{2}$. Then,  for every measurable subset $K$ of $\mathbb{R}^{\N}$, there exist a constant $A_\p$ and for all $f\in L^2(\mathbb{R}^{\N})$ such that
\begin{eqnarray*}
\int_K|(\mathfrak{F}_{\bm\alpha,\lambda}f)(\boldsymbol\eta)|^2d\boldsymbol\eta\leq\frac{|c(\bm\alpha_\lambda)|^2}{C_{\bm\alpha,\lambda}}A_\p\frac{|\sin\bm\alpha|_m}{(\lambda^2)^{\N}}\lambda(K)\int_{\mathbb{R}_0^{\N}}\|(W_{\psi}^{\bm\alpha,\lambda}{f})(\bm{a},\cdot)\|^{2-\frac{\N}{\p}}_{L^2(\mathbb{R}^{\N})}\|\|\cdot\|^\p(W_{\psi}^{\bm\alpha,\lambda}{f})(\bm{a},\cdot)\|_{L^2(\mathbb{R}^{\N})}^{\frac{\N}{\p}}\frac{d\bm{a}}{|\bm{a}|_m^2}.
\end{eqnarray*}
\end{theorem}
\begin{proof}
Using equation (\ref{P1eqn3LUP}), the proof follows similar to that of theorem {\ref{P1Th1LUP}}.
\end{proof}
\section {Conclusions}
We offer a novel definition of the wavelet transformation of a function specified in $\mathbb{R}^{\N}$, i.e. MFrWT, and investigate some of its fundamental aspects, such as linearity, anti-linearity, parity, conjugation, and so on, in this study. For the image space of the suggested transformation, we established the inner product relation and inversion formula for MFrWT with reproducible kernel function. Finally, Heisenberg's uncertainty inequality and the logarithmic uncertainty principle are obtained from the relationship between classical FT and MFrFT. The commonalities between different transform domains can be better understood using logarithmic, Heisenberg, and local uncertainty principles. The three uncertainty principles can be useful in the future because they are features of signal processing.
\section {Acknowledgement}
The work is partly supported by Council of Scientific and Industrial Research (CSIR), New Delhi, India (File No. 09/1023(0035)/2020-EMR-I) and UGC File No. 16-9 (June 2017)/2018(NET/CSIR), New Delhi, India.
\bibliography{MasterNavneet}
\bibliographystyle{plain}
\end{document}